\documentclass[10pt]{elsarticle}
\journal{***}
\usepackage{amssymb,amsmath,amsthm,bm}
\usepackage[left=3cm, right=3cm, top=3.5cm, bottom=3.5cm]{geometry} 
\usepackage{amsfonts}
\usepackage{epsfig}
\usepackage{amsbsy}
\usepackage{color}
\biboptions{sort&compress}
\usepackage{subfigure}
\usepackage{graphicx}
\graphicspath{{figures/}}

\newlength\imagewidth
 \newtheorem{lemma}{Lemma}[section]
 \newtheorem{theorem}{Theorem}[section]
 
 \newtheorem{definition}{Definition}[section]
\newtheorem{remark}{Remark}[section]

\usepackage{color}
\definecolor{dgreen}{rgb}{0,.6,0}
\usepackage[normalem]{ulem}
\begin{document}

\begin{frontmatter}
\title{Spectral analysis  of small amplitude periodic  $b$-Novikov equation under transverse perturbations}

\author{Lin Lu }

\author{Xiaokai He}
\cortext[mycorrespondingauthor]{Corresponding author.}

\author{Aiyong Chen \corref{mycorrespondingauthor}}
\ead{aiyongchen@163.com}

\address{Hunan Provincial University Key Laboratory for Big Data Analysis and Application,
 \\ Hunan First Normal University,  Hunan 410205, China}

\begin{abstract}

This paper is devoted to the transverse stability problem for small-amplitude periodic traveling waves of the 
$b$-Novikov equation, which arises as a two-dimensional extension of the 
$b$-family hierarchy. The perturbations under consideration fall into two groups. One group matches the fundamental period of
 the underlying wave along the longitudinal direction, 
 while the other consists of profiles that are bounded or localized in the transverse direction.
 We perform spectral analysis of the associated linearized operator and derive precise stability and instability
  thresholds. The resulting conditions exhibit  intricate dependence on two key parameters including 
   the cubic coupling strength $b$  and the wavenumber $k$.

\end{abstract}

\begin{keyword}
$b$-Novikov equation; Transverse perturbations; Periodic traveling waves
\end{keyword}

\end{frontmatter}

\section{Introduction}
\numberwithin{equation}{section}

Nonlinear dispersive partial differential equations (PDEs) serve as fundamental mathematical models for
 describing complex wave propagation phenomena in fluid mechanics, plasma physics, nonlinear optics and marine engineering.
 Spectral stability of periodic traveling waves governs the long-time evolution of wave profiles,
  which is of great significance in mathematical physics. For multi-dimensional systems, transverse perturbations deserve particular attention,
  since one-dimensionally stable periodic waves may deform, break or collapse under two-dimensional transverse disturbances,
   an effect invisible to purely one-dimensional stability analysis.

The Kadomtsev–Petviashvili (KP) family, serving as the canonical two-dimensional extension of classical dispersive PDEs,
 has received extensive attention in transverse spectral stability studies.
 Alexander et al. \cite{Alexander1997} proved that KP solitary waves suffer long-wavelength transverse instability in the positive dispersion regime with a finite unstable short-wavelength cutoff, while remaining stable for negative dispersion.
  Kataoka et al. \cite{Kataoka2000} later showed generalized KP solitary waves can develop transverse instability in 
  negative dispersion under appropriate nonlinearity.
Within the generalized KP framework, Johnson and Zumbrun \cite{Johnson2010SIAM} analyzed transverse spectral instability of 
periodic gKdV waves and introduced a geometric orientation index to detect unstable modes, a technique subsequently
 extended to Zakharov–Kuznetsov and Benjamin–Bona–Mahony equations \cite {Johnson2010Stud}.
 Mizumachi and Tzvetkov \cite{Mizumachi2012} established rigorous nonlinear stability for KP-II line solitons under periodic transverse forcing. 
 For fifth-order water-wave KP models, Haragus and Wahlen \cite{Haragu2017} linked transverse instability of periodic and line solitary waves to their unstable periodic far-field tails. 
 Chen and Jin \cite{ChenRM2021} extended this approach to two-dimensional CH-KP-I line solitons,
  and derived linear instability conditions under periodic transverse perturbations.
The development of transverse spectral analysis has yielded refined stability criteria for various KP equations and perturbation types.
 Haragus \cite{Haragu2011} compared the transverse stability of small-amplitude periodic waves in the KP‑I and KP‑II settings.
Chen et al. \cite{Chen2024MM} quantified the combined influence of dispersion parameters and wave numbers on the stability of perturbed KP periodic waves.
Bhavna et al. \cite{Bhavna2022} provided stability theorems for KP-fKdV, KP-ILW and KP-Whitham periodic waves under two-dimensional disturbances. 
Beyond KP systems, transverse instability theory has been applied to NLS rogue waves \cite{Ablowitz2021}, two-dimensional NLS deep-water solitary waves \cite{Deconinck2006}, finite and infinite-depth Stokes water waves \cite{Creedon2026}, and laser-driven plasma waves \cite{Wan2020PRL}.

The Camassa–Holm (CH) integrable equations have attracted recent interest due to their distinctive wave-breaking, peakon and Hamiltonian structures absent from standard KdV models.
 Constantin et al. \cite{Constantin2001PRS,Constantin2006IP,Constantin2009ARMA,Constantin1998AM,Constantin1998ASNSP} established foundational results 
 on CH well-posedness, wave breaking, global existence and scattering for shallow-water solutions. 
 Geyer et al. \cite{Geyer2024} examined smooth solitary waves  within two-dimensional generalized CH systems, proving linear transverse stability for small-amplitude pulses and identifying eigenvalue resonance splitting consistent with KP-II stability theory. Nevertheless, existing literature on CH equations largely addresses one-dimensional spectral analysis or transverse stability of solitary waves, while systematic investigations into periodic traveling waves for cubic nonlinear generalized CH models remain limited.

Related stability analyses for other generalized CH-type models have also been carried out. For the generalized Fornberg–Whitham equation,
 the modulational instability of small-amplitude periodic traveling waves was investigated via spectral perturbation theory, 
 yielding an instability index that depends on the nonlinear parameter and the underlying wavenumber \cite{Zhao2026MM}.
  For the Dullin–Gottwald–Holm (DGH) equation, He et al.  \cite{He2025JMP} established the existence of smooth periodic traveling solutions, 
  proved monotonicity of the period function, and derived a spectral stability criterion within a functional-analytic framework, 
  further demonstrating orbital stability under suitable conditions. 
  These studies highlight the growing interest in stability problems for generalized CH-type dispersive PDEs 
  and motivate the present investigation.

This work focuses on the one-parameter cubic integrable $b$-Novikov equation \cite{b-novikov2013}
\begin{equation}\label{y1N}
u_t-u_{xxt} -buu_xu_{xx} + (b+1)u^2u_x - u^2u_{xxx}=0,\quad t>0, ~x\in\mathbb{R},
\end{equation}
where the real coefficient $b$ tunes cubic nonlinear coupling, recovering the classical Novikov equation at $b=3$ \cite{novikov2009}.
Existing studies on the Novikov and generalized $b$-Novikov equations involve well-posedness, blow-up, peakon solitons, symmetries and one-dimensional spectral theory. Tiglay \cite{Tiglay2011} proved local well-posedness and global existence of analytic periodic solutions. 
 Yan et al. \cite{yan etal2013} extended the well-posedness results to the Cauchy problem on the entire real line. 
Grayshan \cite{Grayshan2013} showed that peakon solution maps lack global uniform continuity in low-order Sobolev spaces.
 Wu and Yin \cite{wu2011} constructed global weak solutions  to the equation,
  and Jiang and Ni \cite{Jiang2012} characterized finite-time blow-up driven by nonlinear wave breaking. 
  Palacios \cite{Palacios2020} derived orbital asymptotic stability for Novikov peakons under momentum
   perturbations from nonnegative Radon measures.
For the generalized $b$-Novikov model, Grayshan and Himonas \cite{gkbch2013} classified power-law CH-type equations and 
derived explicit peakon traveling waves with square-root scaling. Himonas and Holliman \cite{Himonas AA2022} established 
ill-posedness and nonuniqueness for $H^s$ initial data with $s<3/2$ and $b>2$ via colliding bimodal solutions. 
 Silva and Freire \cite{da Silva PL2015} completed Lie symmetry classification, closed-form exact solutions
 and conservation laws for modified Novikov systems.  
 Efstathiou and Petropoulou \cite{Efstathiou2022} constructed high-precision approximations of peakon dynamics
  using the homotopy analysis method,  
 and Hone et al. \cite{Hone etal2009} derived exact multi-peakon formulas for the Novikov equation.
More recent spectral investigations include critical Besov space well-posedness by Zhou and Chen \cite{Zhou S2013}, 
low-regularity ill-posedness by Himonas and Holliman \cite{Himonas2012},
 linear peakon instability proven by Deng and Lafortune \cite{dengxijun2025}, 
 and one-dimensional modulational stability criteria for small-amplitude Novikov periodic waves by Ehrman et al. \cite{Ehrman2025}. 
 Fan et al. \cite{fanlili2025} additionally analyzed modulational instability for quadratic $b$-family periodic waves.

Existing studies on the $b$-Novikov equation have largely been confined to one-dimensional settings,
 including well-posedness, peakon solutions, symmetry properties, and spectral stability for longitudinal disturbances.
  By contrast, the corresponding transverse stability problem for periodic wave trains in the 
  full two-dimensional sense has not yet been systematically investigated.
 As a cubic integrable bridge between the classical Novikov equation and generalized CH systems,
  the $b$-Novikov equation possesses unique nonlinear couplings whose transverse response cannot
   be extrapolated from established KP or CH theories. The joint influence of parameter $b$,
    transverse dispersion and longitudinal wavenumber on spectral stability remains unquantified,
     and a dedicated spectral framework for transverse perturbation problems is absent.  
To evaluate transverse spectral instability of its periodic wave trains, we construct a $(2+1)$-dimensional KP-type generalization
\begin{equation}\label{y1.3}
\Big[ u_t-u_{xxt} - buu_xu_{xx} + (b+1)u^2u_x - u^2u_{xxx} \Big]_x - u_{yy} = 0,
\end{equation}
with $u=u(x,y,t)$ depending on time $t$, longitudinal coordinate $x$ and transverse coordinate $y$. 
Following standard KP construction, we differentiate the full one-dimensional $b$-Novikov operator 
in $x$ and add the transverse dispersive term $- u_{yy}$ to model lateral wave motion. 
 Unlike the purely longitudinal one-dimensional model, 
this two-dimensional generalization enables rigorous assessment of how cubic nonlinearity and transverse
 dispersion jointly determine the stability thresholds for small-amplitude periodic traveling waves.

  Motivated by the aforementioned research, this paper conducts a systematic spectral analysis for small-amplitude periodic traveling waves of the 
  $b$-Novikov equation under two-dimensional transverse perturbations.
   Two categories of perturbations are considered: co-periodic disturbances sharing the same period as the longitudinal wave, and non-periodic transverse modes that are bounded and localized.
    Using spectral perturbation theory and the periodic Evans function method, standard tools for transverse stability analysis,
     we characterize the spectral distribution of the linearized operator and derive explicit stability and instability criteria governed by the parameter $b$, the transverse dispersion strength and the longitudinal wavenumber.
      The results enrich the transverse stability studies for cubic nonlinear $b$-Novikov periodic waves,
       extend the analytical tools for multi-dimensional perturbation analysis of generalized CH-type dispersive PDEs, 
       and further advance the theoretical framework for transverse instability of periodic traveling waves.

To analyze transverse instability of small amplitude periodic traveling waves of equation (\ref{y1.3}), we introduce some
 conventions and notations in the following. The space 
\(L^2(\mathbb{R})\) consists of all real or complex-valued, and Lebesgue-measurable function \(f \) on \(\mathbb{R}\) satisfying
\[\|f\|_{L^2(\mathbb{R})}=\left(\int_{\mathbb{R}}|f(x)|^2\mathrm{d}x\right)^{1/2}<+\infty.\]
Similarly, \(L^2(\mathbb{T})\) denotes the space of \(2\pi\)-periodic, real or complex-valued, measurable function  \(f \) on \(\mathbb{R}\) satisfying
\[\|f\|_{L^{2}(\mathbb{T})}=\left(\frac{1}{2\pi}\int_{0}^{2\pi}|f(x)|^{2}\mathrm{d}x\right)^{1/2}<+\infty.\]
The space $C_{\mathrm{bdd}}(\mathbb{R})$ denotes the set of all bounded continuous functions on $\mathbb{R}$, equipped with the uniform norm
\[
\|f\|_{C_{\mathrm{bdd}}} = \sup_{x\in\mathbb{R}} |f(x)|.
\]
For any real number $s\in\mathbb{R}$, the Sobolev space $H^s(\mathbb{R})$ is defined as the collection of tempered distributions $f$ satisfying
\[
\|f\|_{H^s(\mathbb{R})} = \left( \int_{\mathbb{R}} \big(1+|t|^2\big)^s \big|\widehat{f}(t)\big|^2 \mathrm{d}t \right)^{\frac12} < +\infty,
\]
where $\widehat{f}$ stands for the Fourier transform of $f$.
We further define the periodic Sobolev space on the torus $\mathbb{T}$ by
\[
H^s(\mathbb{T}) = \big\{ f\in H^s_{\mathrm{loc}}(\mathbb{R}) \,\big|\, f \text{ is } 2\pi\text{-periodic} \big\}.
\]
For $f \in L^2(\mathbb{T})$, the Fourier series of $f$ is defined as
\[
f(x) = \sum_{n\in\mathbb{Z}} \widehat{f}_{n} \mathrm{e}^{-\mathrm{i}n x},
\]
where $\widehat{f}_{n}=\frac{1}{2\pi}\int_{0}^{2\pi}f(x)\mathrm{e}^{\mathrm{i}n x}\mathrm{d} x.$
For \(f \in L^2(\mathbb{T})\), its Fourier series converges  to  \(f \) pointwise almost everywhere. The \(L^2(\mathbb{T})\)-inner product is given by
\[\langle f,g\rangle=\frac{1}{2\pi}\int_{0}^{2\pi}f(x)\bar{g}(x)\mathrm{d}x=\sum_{n\in\mathbb{Z}}\widehat{f}_{n}\overline{\widehat{g}}_{n}.\]

The paper is organized as follows. In Section 2, we show the existence of  small-amplitude periodic traveling wave solutions
 of the $b$-Novikov equation. 
 In Section 3, we perform the spectral analysis.
  In Section 4, we consider the periodic perturbation case.
   In Section 5, we study the non-periodic perturbation case, and
   we derive explicit instability criteria and discuss their dependence on the parameters.

\section{Existence of small-amplitude periodic traveling wave solutions}
Taking traveling wave transformation \( u(x,y,t) = \varphi(x-ct) \), where \( c > 0 \) is the wave speed, then equation (\ref{y1.3}) becomes
\begin{equation}\label{2.1}
[-c\varphi' + c\varphi''' - b\varphi\varphi'\varphi'' +(b+1)\varphi^2\varphi' -\varphi^2\varphi''' ]^\prime =0,
\end{equation}
which is equivalent to
\begin{equation}\label{2.1-1z}
[(\varphi^2-c)(\varphi-\varphi^{\prime\prime})^{\prime}+b\varphi\varphi^{\prime}(\varphi-\varphi^{\prime\prime})]^\prime=0.
\end{equation}
Integrating the above equation once, and setting integration constant be zero, 
\begin{equation}\label{2.2}
(\varphi^2-c)(\varphi-\varphi^{\prime\prime})^{\prime}+b\varphi\varphi^{\prime}(\varphi-\varphi^{\prime\prime})=0.
\end{equation}
In our previous work \cite{LUPHD2026}, we have carried out a complete study of the potential function, phase portraits
and existence of periodic solutions for equation (\ref{2.2}). 
That work also clarifies that we impose the constraint \(b>0\) to investigate the  stability of periodic traveling waves. 
We therefore omit these discussions here, 
and all subsequent analyses in the present paper are performed under the assumption \(b>0\).

Let \( \varphi \) be a \( 2\pi/k \)-periodic function of its variable, for some wave number \( k > 0 \) . 
We introduce the scaling \( z = kx \) so that \( w(z) := \varphi(x) \) becomes a \( 2\pi \)-periodic function satisfying
\begin{equation}\label{y2.2}
-cw'+ck^2w'''=bk^2ww'w''-(b+1)w^2w'+k^2w^2w'''
\end{equation}
or, equivalently,
\begin{equation}\label{y-2.3}
(w-k^2w^{\prime\prime})(c-w^2)^{b/2}=d.
\end{equation}
Define  \( F: H^2(\mathbb{T}) \times \mathbb{R}_+ \times \mathbb{R}_+ \times \mathbb{R}_+ \to L^2(\mathbb{T}) \) by
\begin{equation*}
F(w;k,d,c)=(w-k^2w^{\prime\prime})(c-w^2)^{b/2}-d,
\end{equation*}
then solutions to equation (\ref{y2.2}) correspond to solutions \( w\) of
\begin{equation}\label{ODE2.4}
F(w;k,d,c)=0.
\end{equation}
Note that equation (\ref{ODE2.4}) remains invariant under the transformations \( z \to -z \) , \( z \to z+z_0 \) for all \( z_0\in\mathbb{R} \). Therefore, we assume that \( w \) is an even function with respect to \( z \). By the implicit function theorem, for some \( c > 0 \), as long as the operator
\begin{equation*}
\partial_w F(w_0; k, d, c) =1-\frac{bw_0^2}{c-w_0^2} -k^2 \partial_z^2
\end{equation*}
fails to be an isomorphism from \( H^2(\mathbb{T}) \) to \( L^2(\mathbb{T}) \), then non-constant solutions of (\ref{ODE2.4})
 may bifurcate from \( w = w_0 \).   We note that
\[
\partial_w F(w_0; k, d, c) \cos(nz) =\bigg(1-\frac{bw_0^2}{c-w_0^2} +k^2 n^2\bigg) \cos nz.
\]
Thus, when
\begin{equation}\label{y-2.6}
c=c_0=\bigg(\frac{k^2+b+1}{k^2+1}\bigg)w_0^2,
\end{equation}
we have \( \cos(z) \in \ker\left(\partial_w F(w_0; k, d, c)\right) \).
Moreover, since the function
\[
n \mapsto 1-\frac{bw_0^2}{c_0-w_0^2} +k^2 n^2 \in \mathbb{R} \quad \text{for all } n \in \mathbb{N}
\]
is strictly increasing in \( n \), it follows that
\[ \ker\left(\partial_w F(w_0; k, d, c_0)\right) = \text{span}\{\cos z\}. \]
Also, the equilibrium solution $w_0$ satisfies
\begin{equation}\label{2-w0}
\begin{pmatrix}
c-w_0^2
\end{pmatrix}^{b/2}w_0=d,
\end{equation}
substituting $c = c_0$ yields a closed-form expression
\begin{equation}\label{y2.7}
w_0=d^{\frac{1}{b+1}}\bigg(\frac{b}{1+k^2}\bigg)^{-\frac{b}{2(b+1)}}
\end{equation}
and
\begin{equation}\label{y2.8}
c_0=d^{\frac{2}{b+1}}\bigg(\frac{b}{1+k^2}\bigg)^{-\frac{b}{b+1}}\bigg(\frac{k^2+b+1}{k^2+1}\bigg).
\end{equation}
 Using the Lyapunov-Schmidt method, we construct a one-parameter family of non-constant, even and smooth solutions 
 to (\ref{ODE2.4}) near \( w = w_0(k,d) \) and \( c = c_0(k,d)\). Their small-amplitude expansions are given below.

\begin{lemma}\label{lem2.1wc}
For each \( k > 0 \), \( d> 0 \), there exists a family of small-amplitude \( 2\pi/k \)-periodic traveling wave solutions to (\ref{y1.3}) of the form
\[
w(a,d,k) := u\bigl(k(x - c(a,d,k)t)\bigr) 
\]
for $|a|\ll1,$ where \( w \) and \( c \) depend analytically on $k$ and \( a \). The function \( w \) is a smooth, \( 2\pi \)-periodic and even function with respect to \( z \), and \( c \) is an even function with respect to \( a \). Moreover, as \( a \to 0 \), the following asymptotic expansion holds
\begin{equation}\label{y2.9}
w(z;a,d,k)=w_0(k,d)+a\cos z+a^2\big(e_1+e_2\cos2z \big)+  a^3 e_3 \cos3z + O\left(a (a^3+d) \right),
\end{equation}
\begin{equation}\label{y2.10}
c(a,d,k)=c_0(k,d)+a^2c_2+O\left( a (a^3+d) \right),
\end{equation}
with
\begin{equation}\label{y2.11}
e_{1}=\frac{\big[-2(1+k^{2})^{2}+(-4-6k^{2}+k^{4})b+(-2-2k^{2}+k^{4})b^{2}\big]\big(\frac{b}{1+k^{2}}\big)^{\frac{-(2+b)}{2(1+b)}}}{24(1+b)d^\frac{1}{1+b}k^{2}},
\end{equation}
\begin{equation}\label{zje2}
e_{2}=\frac{\big[(2+2b)+(4+3b)k^{2}+(2+b)k^{4}\big]\big(\frac{b}{1+k^{2}}\big)^{\frac{b}{2(1+b)}}}{12bd^{\frac{1}{1+b}}k^{2}},
\end{equation}
\begin{equation}\label{zje3}
e_{3}=  \frac{ \left(\frac{b}{k^2+1}\right)^{-\frac{b+2}{b+1}} \left((b+2) (b+3) k^4+4 (b+1) (b+2) k^2+2 (b+1)^2\right) }{96 k^4  d^{\frac{2}{b+1}}  },
\end{equation}
\begin{equation}\label{y2.12}
c_{2}=\frac{(10+8k^{2}-2k^{4})+(20+12k^{2}+k^{4})b+(10+4k^{2}+k^{4})b^{2}}{12b(1+b)},
\end{equation}
and \( w_0 \) and \( c_0 \) are given by (\ref{y2.7}) and (\ref{y2.8}) respectively. 
\end{lemma}

\begin{proof}
Since \(w\) and \(c\) are analytic in \(a\) and \(c\) is even, we expand
\begin{equation}\label{w2.13}
w(z;a,d,k) = w_0(k, d) + a \cos z + a^2 w_2(z) + a^3 w_3(z) + O(a^4)
\end{equation}
and
\begin{equation}\label{c2.14}
c(a,d,k) = c_0(k, d) + a^2 c_2 + O(a^4),
\end{equation}
where $w_2(z)$ and $w_3(z)$ are \(2\pi\)-periodic even functions. Substituting (\ref{w2.13}) and (\ref{c2.14}) into (\ref{y-2.3})
 and comparing powers of \(a\) yields linear equations at each order.

At \(O(a^2)\), solving for \(w_2=e_1+e_2\cos2z\) gives the coefficients \(e_1,e_2,c_2\) in  \eqref{y2.11}, \eqref{zje2}, \eqref{y2.12}. 
These calculations are detailed in \cite{LUPHD2026} and are omitted here for brevity.

At \(O(a^3)\), after substituting the known \(w_2\) and simplifying, we obtain
\begin{equation}
\begin{aligned}
&-\frac{b \left(k^2+1\right) \left(\frac{b}{k^2+1}\right)^{-\frac{1}{b+1}} \left((b+2) (b+3) k^4
+4 (b+1) (b+2) k^2+2 (b+1)^2\right) \cos 3 z}{4 k^2}\\
&= 3b^2 d^{\frac{2}{1+b}} k^2 (w_3 + w_3''). 
\end{aligned}
\end{equation}
Solving the above equation  with \(w_3=e_3\cos3z\) yields
\[
e_3 = \frac{ (\frac{b}{k^2+1})^{-\frac{b+2}{b+1}} \big((b+2)(b+3)k^4+4(b+1)(b+2)k^2+2(b+1)^2\big) }{96 k^4 d^{\frac{2}{b+1}} },
\]
which is the coefficient not present in \cite{LUPHD2026}. The proof is complete.
\end{proof}

\section{Spectral stability and spectral decomposition}

In this section, let \( w(z;a,d,k) \) with \(d,k>0\) and \(|a| \ll 1\) be a small-amplitude, $2\pi$-periodic traveling wave solution to equation (\ref{y1.3}), whose existence was guaranteed by Lemma 2.1. Linearizing  equation (\ref{y1.3}) at its one-dimensional periodic
 traveling wave solution \( w \) given in (\ref{y2.9}), and considering the perturbation term \( w(z) + \varepsilon V(z,t,y) \), 
 yields the equation 
\begin{equation}
k\Big( \left(1-k^2\partial_z^2\right) V_t - k\mathcal{L}[w]V \Big)_z - V_{yy}=0,
\end{equation}
where
\begin{equation}\label{tag3.2}
\begin{aligned}
\mathcal{L}[w] :=&c\partial_{z}-ck^{2}\partial_{z}^{3}-2(1+b)ww_{z}-(1+b)w^{2}\partial_{z} \\
 & +bk^2(w_zw_{zz}+ww_{zz}\partial_z+ww_z\partial_z^2)+k^2\left(w^2\partial_z^3+2ww_{zzz}\right).
\end{aligned}
\end{equation}
Using change of variable and with a slight abuse of notation, we perform the rescaling
 $t\rightarrow kt, y\rightarrow ky,$
we have 
\begin{equation}
\begin{aligned}
& \Big( \left(1-k^2\partial_z^2\right) V_t -  c\partial_{z}V + ck^{2}\partial_{z}^{3}V + 2(1+b)ww_{z}V + (1+b)w^{2}\partial_{z}V \\
 & -b k^2 (w_zw_{zz}V + ww_{zz}\partial_zV + ww_z\partial_z^2V) - k^2 \left(w^2\partial_z^3V + 2ww_{zzz}V \right) \Big)_z - V_{yy}=0.
\end{aligned}\end{equation}
Setting \( V(z, t, y) = e^{\lambda t + i q y } v(z) \) with \( i^2 = -1 \), we derive 
\begin{equation}
\begin{aligned}\label{xzj3.5}
& \Big( \left(1-k^2\partial_z^2\right) \lambda v -  c\partial_{z} v + ck^{2}\partial_{z}^{3} v + 2(1+b)ww_{z} v + (1+b)w^{2}\partial_{z} v \\
 & -b k^2 (w_zw_{zz}v + ww_{zz}\partial_z v + ww_z\partial_z^2 v) 
 - k^2 \left(w^2\partial_z^3 v + 2ww_{zzz} v \right) \Big)_z + q^2 v=0,
\end{aligned}\end{equation}
which  is
\begin{equation}\label{y3a1}       
 \Big( \left(1-k^2\partial_z^2\right) \lambda v - \mathcal{L}[w] v \Big)_z + q^2 v=0.
\end{equation}
The left-hand side of this equation defines the differential operator
\begin{equation}\label{y3aq2} 
\mathcal{W}_a(\lambda,q)v := 
(1 - k^2\partial_z^2)\partial_z\Bigl[ \lambda v - c v_z + (1 - k^2\partial_z^2)^{-1} \bigl( \widetilde{Q}(v) \bigr) \Bigr] + q^2 v,
\end{equation}
where
\begin{equation*}
\begin{aligned}
\widetilde{Q}(v) = & - b k^2 \big(  w_z w_{zz} v + w w_{zz} \partial_z v + w w_z \partial_z^2 v \big) \\
& + (b+1)  \big( 2w  w_z v+ w^2 \partial_z v \big)  - k^2 \big ( 2w  w_{zzz} v + w^2 \partial_z^3 v \big).
\end{aligned}
\end{equation*}

This work focuses on two-dimensional transverse perturbations corresponding to non-vanishing transverse wavenumber $q$.
 Three distinct perturbation categories are investigated throughout the analysis.
  \begin{enumerate}[(i)]
    \item \textit{z-periodic perturbations}: the operator $\mathcal{W}_a(\lambda,q)$ is defined as a mapping from $H^4(\mathbb{T})$ to $L^2(\mathbb{T})$.
    \item \textit{localized perturbations}: the operator $\mathcal{W}_a(\lambda,q)$ acts from $H^4(\mathbb{R})$ to $L^2(\mathbb{R})$.
    \item \textit{bounded perturbations}: the operator $\mathcal{W}_a(\lambda,q)$ maps $C^4_{\text{bdd}}(\mathbb{R})$ 
    into $C_{\text{bdd}}(\mathbb{R})$.
  \end{enumerate}

The rigorous definition of transverse spectral stability is formulated below.

\begin{definition}[Transverse spectral stability]
Let $u(x,y,t)=w(k(x-ct))$ be a $2\pi/k$-periodic traveling wave solution of equation \eqref{y1.3}, 
where $w$ and $c$ are as defined in \eqref{y2.9} and \eqref{y2.10}, respectively.
 %For each $\lambda\in\mathbb C$ with $\mathrm{Re}(\lambda)>0$ and any parameter $q\neq 0$, 
% define the operator $\mathcal{W}_a(\lambda,q)$. 
 The traveling wave $w$ is said to be transversely 
 spectrally stable if the operator $\mathcal{W}_a(\lambda,q)$ is invertible on the specified function spaces, 
 according to the type of perturbations considered. More precisely,
\begin{enumerate}[(i)]
\item For periodic perturbations, 
\[
\mathcal{W}_a(\lambda,q): H^4(\mathbb{T}) \to L^2(\mathbb{T})
\]
is invertible.
\item For non-periodic localized perturbations,
\[
\mathcal{W}_a(\lambda,q): H^4(\mathbb{R}) \to L^2(\mathbb{R})
\]
is invertible.
\item For non-periodic bounded perturbations,
\[
\mathcal{W}_a(\lambda,q): C_{\mathrm{bdd}}^4(\mathbb{R}) \to C_{\mathrm{bdd}}(\mathbb{R})
\]
is invertible.
\end{enumerate}
The above invertibility conditions are required to hold for all $\lambda\in\mathbb C$
 with $\mathrm{Re}(\lambda)>0$ and for every $q\neq 0$.
\end{definition}

\section{Periodic perturbations}

In this section we study the transverse spectral stability of the periodic waves $w$ with respect to two-dimensional perturbations
 that are periodic in $z$ with the same period $2\pi$.  We investigate the invertibility of $\mathcal{W}_a(\lambda,q)$ defined
  in (\ref{y3aq2}) acting in $L^2(\mathbb{T})$ with domain $H^4(\mathbb{T})$, for $\lambda\in\mathbb{C}$
   with $\operatorname{Re}( \lambda )>0$ and $q\neq0$.

\subsection{Reduction to a spectral problem for $\mathcal{M}_a(q)$}

For $q\neq0$, integrating the equation $\mathcal{W}_a(\lambda,q)v=0$ over one period yields
\[
q^2\int_0^{2\pi}v(z)\,dz =0,
\]
hence any eigenfunction must have zero mean.  We therefore restrict to the subspace
\[
L_0^2(\mathbb{T})=\Bigl\{f\in L^2(\mathbb{T}):\int_0^{2\pi}f(z)\,dz=0\Bigr\},
\]
on which $\partial_z$ is invertible,  with the property that for any $f \in L_0^2(\mathbb{T}),$
\[
\partial_z^{-1}\partial_z f = f - \langle f\rangle,
\]
where
\[
\langle f \rangle := \frac{1}{2\pi}\int_0^{2\pi} f(z)\,dz,
\]
and $\mathcal{P}_0 f := \langle f\rangle$ denotes the orthogonal projection onto constants.

Applying $\partial_z^{-1}$ to  (\ref{y3a1}), and using the fact that  $v \in L_0^2(\mathbb{T}),$ we obtain
\[
(1-k^2\partial_z^2)\lambda v - \mathcal{L}[w]v - \Big\langle (1-k^2\partial_z^2)\lambda v - \mathcal{L}[w]v \Big\rangle = -q^2\partial_z^{-1}v.
\]
Since $v\in L_0^2$, $\langle (1-k^2\partial_z^2)\lambda v\rangle=0$, hence
\[
(1-k^2\partial_z^2)\lambda v - \mathcal{L}[w]v + \langle \mathcal{L}[w]v\rangle = -q^2\partial_z^{-1}v.
\]
Thus the  reduced eigenvalue problem is
\begin{equation}\label{eigenpro}
\lambda v = (1-k^2\partial_z^2)^{-1}\bigl(\mathcal{L}[w]v - \langle \mathcal{L}[w]v\rangle\bigr) - q^2(1-k^2\partial_z^2)^{-1}\partial_z^{-1}v. 
\end{equation}
Since
\[
c\langle v'\rangle - ck^2\langle v'''\rangle = 0,
\]
by (\ref{tag3.2}), we have
\[
\begin{aligned}
\langle \mathcal{L}[w]v\rangle
&= -2(1+b)\langle w w' v\rangle - (1+b)\langle w^2 v'\rangle \\
&\quad + bk^2\langle w' w'' v\rangle + bk^2\langle w w'' v'\rangle + bk^2\langle w w' v''\rangle \\
&\quad + k^2\langle w^2 v'''\rangle + 2k^2\langle w w''' v\rangle.
\end{aligned}
\]
Integration by parts over the periodic interval $[0,2\pi]$ yields
\begin{align*}
-(1+b)\langle w^2 v'\rangle 
&= (1+b)\langle (w^2)' v\rangle = 2(1+b)\langle w w' v\rangle,
\end{align*}
which cancels $-2(1+b)\langle w w' v\rangle$.
Since
\begin{align*}
bk^2\langle w w'' v'\rangle &= -bk^2\langle (w w'')' v\rangle = -bk^2\langle w' w'' v\rangle - bk^2\langle w w''' v\rangle,\\
bk^2\langle w w' v''\rangle &= bk^2\langle (w w')'' v\rangle = bk^2\langle 3 w' w'' + w w''' v\rangle,
\end{align*}
\[
k^2\langle w^2 v'''\rangle = -k^2\langle (w^2)''' v\rangle 
= -k^2\langle 6 w' w'' + 2 w w''' v\rangle,
\]
 we obtain 
\[
\langle \mathcal{L}[w]v\rangle = 3(b-2)k^2\,\langle w' w'' v\rangle,
\]
or equivalently,
\[
\mathcal{P}_0(\mathcal{L}[w]v) = \frac{3(b-2)k^2}{2\pi}\int_0^{2\pi} w'(z) w''(z) v(z)\,dz.
\]
Substituting the above projection term into (\ref{eigenpro}), we define the 
 spectral operator acting on $L_0^2(\mathbb{T})$ with domain $H^3(\mathbb{T})\cap L_0^2(\mathbb{T})$ as
\begin{equation}\label{tag4.19q1}
\mathcal{M}_a(q) :=
(1-k^2\partial_z^2)^{-1}\Bigl(\mathcal{L}[w] - \mathcal{P}_0\mathcal{L}[w]\Bigr)
- q^2 (1-k^2\partial_z^2)^{-1}\partial_z^{-1}.
\end{equation}
Explicitly,
\begin{equation}\label{tag4.1}
\mathcal{M}_a (q) v =
(1-k^2\partial_z^2)^{-1}\mathcal{L}[w]v
- 3(b-2)k^2 (1-k^2\partial_z^2)^{-1}\langle w' w'' v\rangle
- q^2 (1-k^2\partial_z^2)^{-1}\partial_z^{-1}v.
\end{equation}

\subsection{The unperturbed operator $\mathcal{M}_0(q)$}

For $a=0$, we have $w\equiv w_0$ and $c\equiv c_0$ given by (\ref{y2.7}) and (\ref{y2.8}).
  Substituting the constant solution into (\ref{tag3.2}) and noting that all derivatives of $w$ vanish, we obtain
\[
\mathcal{L}[w_0] = \bigl(c_0-(1+b)w_0^2\bigr)\partial_z + \bigl(-c_0k^2+k^2w_0^2\bigr)\partial_z^3.
\]
Using the bifurcation condition (\ref{y2.8}),  we have
\[
c_0-(1+b)w_0^2  = -\frac{bk^2}{k^2+1}\,w_0^2
\]
and
\[
-k^2c_0+k^2w_0^2  = -\frac{bk^2}{k^2+1}\,w_0^2.
\]
Thus
\begin{equation}\label{tage4.2}
\mathcal{L}[w_0] = -\frac{bk^2}{k^2+1}\,w_0^2\bigl(\partial_z+\partial_z^3\bigr). 
\end{equation}
Inserting (\ref{tage4.2}) into (\ref{tag4.1}) gives
\begin{equation}\label{tage4.3}
\mathcal{M}_0(q) = -\frac{bk^2}{k^2+1}\,w_0^2(1-k^2\partial_z^2)^{-1}(\partial_z+\partial_z^3) - q^2(1-k^2\partial_z^2)^{-1}\partial_z^{-1} \end{equation}
and
\begin{equation}
\mathcal{M}_a(q) - \mathcal{M}_0(q) =  (1-k^2\partial_z^2)^{-1} (\mathcal{L}[w] - \mathcal{L}[w_0])
 -  3(b-2)k^2 (1-k^2\partial_z^2)^{-1}\langle w' w'' \cdot \rangle.
 \end{equation}
 A direct calculation shows that 
 \begin{equation}\label{tagOA_a}
||\mathcal{M}_a(q) - \mathcal{M}_0(q)||_{  H^1 {(\mathbb{T})} \rightarrow L^2 (\mathbb{T}) }=O (|a|), ~~~~ a \rightarrow 0.
 \end{equation}

Since
\begin{equation*}
 -\frac{bk^2}{k^2+1}\,w_0^2(1-k^2\partial_z^2)^{-1}(\partial_z+\partial_z^3) e^{inz}
= i\,\frac{bk^2}{k^2+1}\,w_0^2\,\frac{n(n^2-1)}{1+k^2n^2}e^{inz}
\end{equation*}
and
\[
-q^2(1-k^2\partial_z^2)^{-1}\partial_z^{-1}e^{inz} = i\,\frac{q^2}{n(1+k^2n^2)}e^{inz},
\]
we have
\[
\mathcal{M}_0(q)e^{inz} = i\left( \frac{bk^2}{k^2+1}\,w_0^2\,\frac{n(n^2-1)}{1+k^2n^2} + \frac{q^2}{n(1+k^2n^2)} \right) e^{inz}.
\]
Therefore the eigenvalues are purely imaginary $i\omega_n$ with
\begin{equation}\label{ta4.5}
\omega_n = \frac{n}{1+k^2n^2}\left(bw_0^2 \frac{k^2(n^2-1)}{1+k^2} + \frac{q^2}{n^2} \right),\qquad n\in\mathbb{Z}^*.
\end{equation}
Write $\mathcal{M}_0(q)=\mathcal{J}\mathcal{K}_0(q),$ where $\mathcal{J}:=\partial_z(1-k^2\partial_z^2)^{-1}$ is 
skew‑adjoint on $L_0^2(\mathbb{T})$.
Then $\mathcal{K}_0(q)=\mathcal{J}^{-1}\mathcal{M}_0(q)$.
By using (\ref{tage4.3}),
we have
\[
\begin{aligned}
\mathcal K_0(q)
&= (1 - k^2\partial_z^2)\partial_z^{-1}
\left[
-\frac{b k^2}{k^2+1} w_0^2 (1 - k^2\partial_z^2)^{-1}(\partial_z+\partial_z^3)
- q^2 (1 - k^2\partial_z^2)^{-1}\partial_z^{-1}
\right] \\
&= -\frac{b k^2}{k^2+1} w_0^2 (1+\partial_z^2)
- q^2 \partial_z^{-2}.
\end{aligned}
\]
Since \(\partial_z^2\) and  \(\partial_z^{-2}\) are self-adjoint,
 combinations with real constants and real coefficients preserve self-adjointness,
 we have
$\mathcal K_0(q)$ is a self-adjoint operator on  $L_0^2(\mathbb T).$
Since $\mathcal{M}_0(q)e^{inz} = i \omega_n e^{inz},$  from $\mathcal{J}^{-1}=(1-k^2\partial_z^2)\partial_z^{-1}$ we obtain
\[
\mathcal{K}_0(q)e^{inz}  = (1+k^2n^2)\frac{\omega_n}{n} e^{inz}.
\]
Using (\ref{ta4.5}), 
\[
(1+k^2n^2)\frac{\omega_n}{n} =\left(bw_0^2 \frac{k^2(n^2-1)}{1+k^2} + \frac{q^2}{n^2} \right).
\]
Thus, the Krein signature $K_n$ of the eigenvalue $i\omega_n$ of $\mathcal{M}_0(q)$ is defined as
\begin{equation}\label{ta4.6}
K_n := \operatorname{sgn}\bigl\langle \mathcal{K}_0(q)e^{inz}, e^{inz}\bigr\rangle = \operatorname{sgn}\left(bw_0^2 \frac{k^2(n^2-1)}{1+k^2} + \frac{q^2}{n^2} \right). 
\end{equation}

\subsubsection{Finite and short wavelength transverse perturbations}

For $b>0,$ when $q>0$ and $|n|\geq1$, we have $\left( bw_0^2 \frac{k^2(n^2-1)}{1+k^2} + \frac{q^2}{n^2} \right)>0$.
 Hence all $K_n$  are positive for $|n|\geq1$.  Standard perturbation arguments  show that for $q$ bounded away from the origin, 
 the spectrum of $\mathcal{M}_a(q)$ remains purely imaginary for sufficiently small $|a|$.  Thus, we have the following conclusion. 
 \begin{lemma}
  For any given $q^*> 0$, there exists $|a|$ sufficiently small such that for all
$| q | >  q^*$, the spectrum of $\mathcal{M}_a(q)$ is purely imaginary.
 \end{lemma}
 
 From the above discussion, we can see that the only possible instability arises from long wavelength transverse perturbations, i.e. when
 $q$ is small, the collision may occurs at the origin.
\subsubsection{Long wavelength transverse perturbations}

For $q$ and $a$ small, we regard $\mathcal{M}_a(q)$ as a perturbation of the constant‑coefficient operator $\mathcal{M}_0(0)$. 
 From (\ref{tage4.3}) with $q=0$,
\begin{equation}\label{taA_00}
\mathcal{M}_0(0) = -\frac{bk^2}{k^2+1}\,w_0^2(1-k^2\partial_z^2)^{-1}(\partial_z+\partial_z^3),
\end{equation}
its spectrum in $L_0^2(\mathbb{T})$ is
\[
\operatorname{spec}(\mathcal{M}_0(0)) = \left\{ i\, \frac{b w_0^2 k^2}{(k^2+1)} \frac{n(n^2-1)}{(1+k^2n^2)} 
 : n\in\mathbb{Z}^* \right\}.
\]
In particular, zero is a double eigenvalue (for $n=\pm1$), and all other eigenvalues are simple, purely 
imaginary and lie outside a fixed ball around the origin. 

From (\ref{tag4.1}) and  (\ref{taA_00}), we obtain
\begin{equation}
\begin{aligned}
\mathcal{M}_a(q) - \mathcal{M}_0(0) &=  (1-k^2\partial_z^2)^{-1}\left( \mathcal{L}[w] + \frac{bk^2}{k^2+1}\,w_0^2(\partial_z+\partial_z^3) \right)\\
& ~~~ -  3(b-2)k^2 (1-k^2\partial_z^2)^{-1}\langle w' w'' \cdot \rangle
 - q^2(1-k^2\partial_z^2)^{-1}\partial_z^{-1},
 \end{aligned}
\end{equation}
by  (\ref{tagOA_a}), we have
 \begin{equation}
||\mathcal{M}_a(q) - \mathcal{M}_0(0)||_{  H^1 {(\mathbb{T})} \rightarrow L^2 (\mathbb{T}) }=O (q^2 +| a|).
 \end{equation}
Thus, by a standard perturbation theory, for sufficiently small $|q|$ and $|a|$, the spectrum of $\mathcal{M}_a(q)$ 
splits into a two‑dimensional part $\operatorname{spec}_0(\mathcal{M}_a(q))$ near the origin and the
 rest $\operatorname{spec}_1(\mathcal{M}_a(q))$ which is purely imaginary. 
 Moreover, the spectral projection $\Pi_a(q)$ onto the two‑dimensional subspace $\mathcal{X}_a(q)=\Pi_a(q)L_0^2(\mathbb{T})$ satisfies $\|\Pi_a(q)-\Pi_0(0)\|=O(q^2+|a|)$.

To determine the two eigenvalues in $\operatorname{spec}_0(\mathcal{M}_a(q))$, we compute the $2\times2$ matrix representing the restriction of $\mathcal{M}_a(q)$ to $\mathcal{X}_a(q)$. 
From (\ref{taA_00}), using
$(\partial_z+\partial_z^3)\sin z =0$ and
$(\partial_z+\partial_z^3)\cos z=0,$
it follows that $\mathcal{M}_0(0)\sin z = \mathcal{M}_0(0)\cos z =0$. Hence zero is a double eigenvalue of $\mathcal{M}_0(0)$ with
 eigenspace spanned by $\{\cos z,\sin z\}$.

To find the eigenvalues of $\mathcal{M}_0(q)$ for small $q$, we compute its action on the basis $\{\cos z,\sin z\}$.
By using (\ref{tage4.3}), we obtain
\[
\begin{aligned}
\mathcal{M}_0(q)\cos z &= -\dfrac{bk^2w_0^2}{k^2+1}(1-k^2\partial_z^2)^{-1}(\partial_z+\partial_z^3)\cos z 
  - q^2(1-k^2\partial_z^2)^{-1}\partial_z^{-1}\cos z \\
&= \frac{- q^2}{1+k^2}\sin z,
\end{aligned}
\]
\[
\begin{aligned}
\mathcal{M}_0(q)\sin z &= -\dfrac{bk^2w_0^2}{k^2+1}(1-k^2\partial_z^2)^{-1}(\partial_z+\partial_z^3)\sin z 
  - q^2(1-k^2\partial_z^2)^{-1}\partial_z^{-1}\sin z \\
&= \frac{q^2}{1+k^2}\cos z.
\end{aligned}
\]
Thus, in the basis $\{\cos z,\sin z\}$, the matrix representation of $\mathcal{M}_0(q)$ is
\begin{equation}\label{M0ell}
M_0(q)=\begin{pmatrix}
0 & \dfrac{q^2}{1+k^2}\\[6pt]
\dfrac{-q^2}{1+k^2} & 0
\end{pmatrix},
\end{equation}
the eigenvalues are
 $ i\frac{q^2}{1+k^2}$ and $ - i\frac{q^2}{1+k^2}$.
Therefore, for $a=0$,
\[
 \operatorname{spec}_0(\mathcal{M}_0(q)) = \left\{ i\frac{q^2}{1+k^2},\; -i\frac{q^2}{1+k^2} \right\}.
\]

From  (\ref{y2.9}) we have
\[
\partial_z w = -a\sin z -2a^2e_2\sin2z-3a^3e_3\sin3z+O(a^4).
\]
To obtain a vector that tends to $\sin z$ as $a\to0$, we define
\[
 \eta_a^0(z):=-\frac1a\,\partial_z w(z)=\sin z+2a e_2\sin2z+3a^2e_3\sin3z+O(a^3).
\]
Clearly $\eta_a^0$ is odd in $z$ and satisfies $\mathcal{M}_a(0)\eta_a^0=0$.

For $a=0$, the second basis vector of the two‑dimensional eigenspace of $\mathcal{M}_0(q)$  is $\cos z$, which is even. 
We choose an even vector $\eta_a^1(z)$ such that $\eta_a^1\to\cos z$ as $a\to0$.
Set
\[
\eta_a^1(z):=\cos z+a\psi_1(z)+a^2\psi_2(z)+O(a^3),
\]
where $\psi_1$ and $\psi_2$ are even. For $q=0$, the restriction of $\mathcal{M}_a(0)$ to the two‑dimensional subspace $\mathcal{X}_a(0)$ must satisfy
\[
\mathcal{M}_a(0)\eta_a^1(z)=\mu(a)\,\eta_a^0(z),\qquad \mu(a)=O(a^2),
\]
since $\eta_a^0$ spans the kernel of $\mathcal{M}_a(0)$ and the other direction is mapped into the kernel. Using (\ref{y2.9}) and
 (\ref{y2.10}) and the explicit form of $\mathcal{M}_a(0)$  and collecting powers of $a$, we obtain 
\[
\psi_1(z)=2e_2\cos2z,\qquad \psi_2(z)=3e_3\cos3z,
\]
where $e_2$ and $e_3$ are given in (\ref{zje2}) and (\ref{zje3}), respectively.  Thus,
\[
 \eta_a^1(z)=\cos z+2a e_2\cos2z+3a^2e_3\cos3z+O(a^3).
\]

Next, we construct the matrix of $\mathcal{M}_a(0)$ in the basis $\{\eta_a^1,\eta_a^0\}$.
For $q=0$, the operator $\mathcal{M}_a(0)$ is given by
\[
\mathcal{M}_a(0) = (1-k^2\partial_z^2)^{-1}\mathcal{L}[w],
\]
with $\mathcal{L}[w]$ defined in (\ref{tag3.2}).  
Since $\mathcal{M}_a(0)$ maps even functions to odd functions, $\mathcal{M}_a(0)\eta_a^1$ is odd and therefore must be a multiple of $\eta_a^0$.
 Hence the matrix representation of $\mathcal{M}_a(0)$ in the  basis $\{\eta_a^1,\eta_a^0\}$ is of the form
\[
M_a(0)=\begin{pmatrix}
0 & 0\\
m_{10}(a) & 0
\end{pmatrix},
\]
where $m_{10}(a)$ is to be determined.
By calculation, we have
\[
\mathcal{M}_a(0)\eta_a^1 = \left( \frac{(7-b) (b+2) k^4+2 (b+1) (b+8) k^2+2 (b+1)^2}{6 b k^2} a^2   + O (a^3) \right ) \eta_a^0.
\]
Thus,
\[
m_{10}(a) = \frac{(7-b) (b+2) k^4+2 (b+1) (b+8) k^2+2 (b+1)^2}{6 b k^2} a^2   + O (a^3),
\]
and the matrix representation of $\mathcal{M}_a(0)$ in the basis $\{\eta_a^1,\eta_a^0\}$ is
\begin{equation}\label{Ma0}
M_a(0)=\begin{pmatrix}
0 & 0\\[6pt]
\frac{(7-b) (b+2) k^4+2 (b+1) (b+8) k^2+2 (b+1)^2}{6 b k^2} a^2   + O (|a|^3) & 0
\end{pmatrix}.
\end{equation}
From (\ref{M0ell}) and  (\ref{Ma0}), we have
\begin{equation}\label{Maell}
M_a(q)=\begin{pmatrix}
0 & \dfrac{ q^2}{1+k^2}\\[6pt]
\dfrac{-q^2}{1+k^2} + \frac{(7-b) (b+2) k^4+2 (b+1) (b+8) k^2+2 (b+1)^2}{6 b k^2} a^2 + O\Big( |a|(q^2+a^2) \Big) & 0
\end{pmatrix}.
\end{equation}
The eigenvalues of $M_a(q)$ are roots of the characteristic polynomial
\[
\begin{aligned}
& \mathrm{det} ( \lambda I - M_a(q)) \\
&=\lambda^2 + \dfrac{q^2}{1+k^2} \left( \dfrac{q^2}{1+k^2} - \frac{(7-b) (b+2) k^4+2 (b+1) (b+8) k^2+2 (b+1)^2}{6 b k^2} a^2 \right)
 + O\Big( |a|(q^2+a^2) \Big),
\end{aligned}
\]
here $I$ denotes the identity matrix.
Thus, we have
\[
\begin{aligned}
\lambda^2 &= - \dfrac{q^2}{1+k^2} \left( \dfrac{q^2}{1+k^2} - \frac{(7-b) (b+2) k^4+2 (b+1) (b+8) k^2+2 (b+1)^2}{6 b k^2} a^2 \right)
 + O\Big( |a|(q^2+a^2) \Big)\\
 &= - \dfrac{q^2}{(1+k^2)^2} \left( q^2 - q_a^2 \right)
 + O\Big( |a|(q^2+a^2) \Big),
\end{aligned}
\]
where
\begin{equation}\label{ella2}
q_a^2 :=    \frac{(7-b) (b+2) k^4+2 (b+1) (b+8) k^2+2 (b+1)^2}{6 b k^2} (k^2+1) a^2.
\end{equation}

From the above analysis, we derive critical results on the transverse spectral stability of periodic  perturbations.
\begin{theorem}[Transverse instability for $b$-Novikov]\label{thm:NI}
 Assume that $|a|$ and $|q|$ are sufficiently small.  Define $q_a^2$ by (\ref{ella2}).  

If $q_a^2>0$, then there exists $q_*^2 = q_a^2+O(a^4)>0$ such that
  \begin{enumerate}[(i)]
    \item for $q^2>q_*^2$ the spectrum of $\mathcal{M}_a(q)$ is purely imaginary (transverse spectral stability);
    \item for $q^2<q_*^2$ the spectrum contains a pair of real eigenvalues of opposite signs (transverse spectral instability).
  \end{enumerate}
  
 If $q_a^2<0$, then the spectrum of $\mathcal{M}_a(q)$ is purely imaginary.

\end{theorem}

Next, we present a lemma for determining the sign of $q_a^2$.
\begin{lemma}
When $0<b\le 7$, we have $q_a^2>0$. When $b>7$ and  $0<k^2<p_0(b)$, we have $q_a^2>0$,
 where 
\begin{equation}\label{x0b}
p_0(b):=\frac{(b+1)\Big[ (b+8)+ \sqrt{ 3(b+1)^2+33 } \Big]}{(b-7)(b+2)}.
\end{equation}
When $b>7$ and   $k^2>p_0(b)$, we have $q_a^2<0$.
 
\end{lemma}

\begin{proof}

Since $b>0$,  $q_a^2$ shares the same sign as
$$
(7-b)(b+2)k^4+2(b+1)(b+8)k^2+2(b+1)^2.
$$
Set $p=k^2>0$ and denote 
\[
N(p):=(7-b)(b+2)p^2+2(b+1)(b+8)p+2(b+1)^2.
\]

 When  $0<b\le 7$,  $(7-b)(b+2)\ge 0$ and $2(b+1)(b+8)>0$,  $N(p)>0$ for all $p>0$. Hence  $q_a^2>0$.

 When $b>7$, $(7-b)(b+2)<0$,  the quadratic $N(p)$ opens downward. Since the discriminant is
$
4(b+1)^2\Bigl[3(b+1)^2+33\Bigr]>0,
$
 $N(p)$ has two real roots, one positive and one negative. The positive root is
\[
p_0(b)=\frac{(b+1)\Big[ (b+8)+ \sqrt{ 3(b+1)^2+33 } \Big]}{(b-7)(b+2)}.
\]
 If $0<p<p_0(b)$, then $N(p)>0$, $q_a^2>0$.
 If $p>p_0(b)$, then $N(p)<0$,  $q_a^2<0$.

\end{proof}

\begin{remark}
If the term $-u_{yy}$ in equation (\ref{y1.3}) is replaced by $u_{yy}$,
the resulting equation can be called as the 
$b$-Novikov-II equation. Its spectral stability under periodic perturbations can be investigated using an analogous approach. 
Since the procedure is largely parallel to that for equation (\ref{y1.3}), we omit the details here.
\end{remark}

\begin{remark}

In our formulation, the operator $\mathcal{M}_a(q)$ incorporates the projection $\mathcal{P}_0$ 
via a mean-value correction term of the form
\[
-3(b-2)k^2\bigl(1-k^2\partial_z^2\bigr)^{-1}\langle w' w'' v\rangle,
\]
where $\langle\cdot\rangle$ denotes the period average over one spatial period.

     Since the periodic traveling wave profile $w(z)$ is an even function of $z$, it follows that  $w'w''$ is odd.
      Since $\eta_a^1$ is even, the product $w'w''\eta_a^1$ 
    remains an odd function. The period average of any odd periodic function over a full period equals zero, i.e.,
    \[
    \langle w'w'' \eta_a^1 \rangle = 0.
    \]
    Therefore the projection correction term contributes nothing to $\mathcal{M}_a(q)\eta_a^1$.
    
    Next consider the odd basis vector $\eta_a^0$, which spans the kernel of $\mathcal{M}_a(0)$, i.e., $\mathcal{M}_a(0)\eta_a^0 = 0$. By definition,
\[
\mathcal{M}_a(0) = \bigl(1-k^2\partial_z^2\bigr)^{-1}\big(\mathcal{L}[w] - \mathcal{P}_0\mathcal{L}[w]\big).
\]
Combining $\mathcal{M}_a(0)\eta_a^0 = 0$ with the identity $\mathcal{L}[w]\eta_a^0 = 0$, we  obtain $\mathcal{P}_0\mathcal{L}[w]\eta_a^0 = 0$. Therefore the projection term also has no effect on $\eta_a^0$.

The key matrix entry $m_{10}(a)$, which encodes the finite-amplitude correction to the unperturbed spectral problem, is computed by projecting $\mathcal{M}_a(0)\eta_a^1$ onto $\eta_a^0$. As shown above, the projection operator $\mathcal{P}_0$ yields zero on both basis vectors of $\mathcal{X}_a(q)$, so it does not modify the expression of $m_{10}(a)$. Consequently, the critical transverse wavenumber threshold $q_a^2$ is fully independent of $\mathcal{P}_0$, and all transverse spectral stability conclusions stated in Theorem \ref{thm:NI} remain unchanged if the projection term is omitted entirely.

\end{remark}

\section{Non-periodic perturbations }
 
This section focuses on two-dimensional non-periodic perturbations, which may take localized or bounded forms along the 
 propagation direction of waves. 
We examine the invertibility of \(\mathcal{W}_a(\lambda,q)\) from equation (\ref{y3aq2}) over the spaces \(L^2(\mathbb{R})\) and \(C_{\mathrm{bdd}}(\mathbb{R})\), with domain spaces \(H^4(\mathbb{R})\) and \(C_{\mathrm{bdd}}^4(\mathbb{R})\) respectively, for all \(\lambda\in\mathbb{C}\) with positive real part and every nonzero real transverse wavenumber \(q\).

\subsection{ Reformulation via Floquet theory and fundamental lemmas}

Given that all coefficients of \(\mathcal{W}_a(\lambda,q)\) are \(2\pi\)-periodic with respect to the variable $z$, 
we apply the classical Floquet–Bloch decomposition method. Every solution $v$ belonging to either \(L^2(\mathbb{R})\) or
 \(C_{\mathrm{bdd}}(\mathbb{R})\) satisfying the homogeneous equation \(\mathcal{W}_a(\lambda,q)v=0\) can be
  expressed as \(v(z) = e^{i\eta z}\phi(z)\), where the Floquet exponent \(\eta\) falls within the 
  interval \((-\tfrac12,\tfrac12]\), and $\phi(z)$ is a $2\pi$-periodic function.

\begin{lemma}[Floquet reduction]
The operator $\mathcal{W}_a(\lambda,q)$ is invertible on $L^2(\mathbb{R})$ if and only if 
for all Floquet exponents $\eta\in(-\frac{1}{2},\frac{1}{2}]$, the associated Bloch operator
\[
\begin{aligned}
    \mathcal{W}_{a, \eta}(\lambda, q)
= \big(1-k^{2} \left(\partial_z + i \eta\right)^{2}\big)\left(\partial_z + i \eta\right)
\Bigg[\lambda -c \left(\partial_z + i \eta\right) + \big(1-k^{2}\left(\partial_z +i \eta\right)^{2}\big)^{-1} \bigl( \widetilde{Q}_\eta \bigr) \Bigg] 
+ q^{2}
\end{aligned}
\]
is invertible on $L^2(\mathbb{T})$ with domain $H^4(\mathbb{T})$, where
\[
\begin{aligned}
\widetilde{Q}_\eta=& - b k^2\big(w_z w_{zz}+w w_{zz}\left(\partial_z+i\eta\right)+w w_z\left(\partial_z+i\eta\right)^2\big)\\
&+(1+b) ( 2w w_z+ w^2 \left(\partial_z+i\eta\right) )- k^2\big( 2 w w_{zzz} + w^2 (\partial_z+i\eta)^3 \big).
\end{aligned}
\]
\end{lemma}

As \(\eta=0\) corresponds to the periodic perturbations investigated  in section 4, we focus  on \(\eta \neq 0\).
For $\eta\neq 0$, the operator $(\partial_z+i\eta)$ is invertible on $L^2(\mathbb{T})$. Applying $(\partial_z+i\eta)^{-1}$ to the equation 
$ \mathcal{W}_{a,\eta}(\lambda,q)v=0$, we obtain the equivalent spectral operator.

\begin{lemma}[Spectral equivalence]
For $\eta\in(-\frac{1}{2},\frac{1}{2}]$ and $\eta \neq 0$, $ \mathcal{W}_{a,\eta}(\lambda,q)$ is non-invertible if and only if $\lambda$
 belongs to the spectrum of
\[
\mathcal{M}_a(q,\eta)
= \big(1-k^2(\partial_z+i\eta)^2\big)^{-1} \mathcal{L}_{ \eta}[w]
- q^2 \big(1-k^2(\partial_z+i\eta)^2\big)^{-1}\big(\partial_z+i\eta\big)^{-1}
\]
acting on $L^2(\mathbb{T})$ with domain $H^3(\mathbb{T})$,
where
\begin{equation}
\begin{aligned}
\mathcal{L}_{ \eta}[w] =&c (\partial_{z}+ i \eta) -ck^{2}(\partial_{z}+ i \eta)^{3}-2(1+b)ww_{z}-(1+b)w^{2}(\partial_{z}+ i \eta)  \\
 & +bk^2(w_zw_{zz}+ww_{zz}(\partial_{z}+ i \eta) +ww_z(\partial_{z}+ i \eta)^2)+k^2\left(w^2(\partial_{z}+ i \eta)^3+2ww_{zzz}\right).
\end{aligned}
\end{equation}
\end{lemma}

As \(\eta\) approaches zero, the operator \((\partial_z + i\eta)^{-1}\) becomes singular. It follows that connections between the spectrum
 of \(\mathcal{M}_a(q,\eta)\) and the invertibility of \(\mathcal{W}_{a,\eta}(\lambda,q)\) lack uniformity across all admissible \(\eta\).
 Hence we focus on the regime \(|\eta|> \epsilon> 0\) to locate the threshold for instability.

\begin{lemma}[Spectral symmetry]
For $\eta\in(-\frac{1}{2},\frac{1}{2}]$ and $\eta \neq 0$, then
the spectrum of $\mathcal{M}_a(q,\eta)$ is symmetric with respect to the imaginary axis. Moreover,
\[
\operatorname{spec}_{L^2(\mathbb{T})}(\mathcal{M}_a(q,\eta))=\operatorname{spec}_{L^2(\mathbb{T})}(-\mathcal{M}_a(q,-\eta)).
\]

\end{lemma}
\begin{proof}
The proof is similar to that in \cite{Chen2024MM}. By using the anti-commutation with the symmetry $\mathcal{S}\psi(z)=\overline{\psi(-z)}$ and the reflection $\mathcal{R}\psi(z)=\psi(-z)$ one obtains the two properties. The evenness of $w$ is essential.
\end{proof}

Thus, if $\lambda$ is an eigenvalue with eigenfunction $\psi$, then $-\overline{\lambda}$ is also an eigenvalue,
  which implies the imaginary-axis symmetry. By symmetry, we only consider $\eta\in(0,\frac{1}{2}]$ in the following analysis.
  It is easy to verify that
for small amplitude $|a|\ll1$, the perturbation term satisfies
\[
\big\|\mathcal{M}_a(q,\eta) - \mathcal{M}_0(q,\eta)   \big\|_{H^1(\mathbb{T})\to L^2(\mathbb{T})} = O(|a|)
\]
as $ a\to 0$ uniformly for all $\eta\in(0,\frac{1}{2}]$, where \(\mathcal{M}_0(q,\eta)\) is the operator associated with the constant solution
  \(w \equiv w_0\) and the corresponding constant wave speed \(c \equiv c_0\).

\subsection{ Spectrum of the unperturbed operator $\mathcal{M}_0(q,\eta)$}
Recall that
\[
c_0 = \frac{k^2+b+1}{k^2+1}w_0^2,\quad
\mathcal{L}[w_0] = -\frac{b k^2}{k^2+1}w_0^2\big(\partial_z+\partial_{z}^3\big).
\]
Replacing $\partial_z \mapsto \partial_z+i\eta$, we have
\[
\mathcal{L}_\eta[w_0] = -\frac{b k^2}{k^2+1}w_0^2\big((\partial_z+i\eta)+(\partial_z+i\eta)^3\big).
\]
Thus
\begin{equation}
\begin{aligned}
\mathcal{M}_0(q,\eta)
&= (-\frac{b k^2}{k^2+1}w_0^2)  \big(1-k^2(\partial_z+i\eta)^2\big)^{-1} \big((\partial_z+i\eta)+(\partial_z+i\eta)^3\big)\\
&\quad - q^2 \big(1-k^2(\partial_z+i\eta)^2\big)^{-1}\big(\partial_z+i\eta\big)^{-1}.
\end{aligned}
\end{equation}
Then we have
\[
\mathcal{M}_0(q,\eta)e^{inz} = i\omega_n(q,\eta)\, e^{inz},
\]
where 
\begin{equation}\label{omegan}
\omega_n(q,\eta)
= \frac{n+\eta}{1+k^2(n+\eta)^2}
\left(
b w_0^2 \cdot \frac{k^2\big((n+\eta)^2-1\big)}{1+k^2}
+ \frac{q^2}{(n+\eta)^2}
\right).
\end{equation}
We decompose $\mathcal{M}_0(q,\eta) = \mathcal{J}_\eta \mathcal{K}_0(q,\eta)$, where
\[
\mathcal{J}_\eta = (\partial_z + i\eta)\big(1-k^2(\partial_z+i\eta)^2\big)^{-1}
\]
is skew-adjoint, and $\mathcal{K}_0(q,\eta)$ is self-adjoint. Define
\begin{equation}\label{munellxi}
\mu_n(q,\eta):=\frac{b w_0^2 k^2}{1+k^2}\bigl((n+\eta)^2-1\bigr)+\frac{q^2}{(n+\eta)^2},
\end{equation}
then $\omega_n=\frac{n+\eta}{1+k^2(n+\eta)^2}\,\mu_n$.
 The Krein signature of the eigenvalue $i\omega_n$ is
\begin{equation}\label{Knxi}
K_{n,\eta}=\operatorname{sgn}\bigl(\mu_n(q,\eta)\bigr).
\end{equation}
Note that
\[
\mu_0=\frac{b w_0^2 k^2}{1+k^2}(\eta^2-1)+\frac{q^2}{\eta^2},\qquad
\mu_{-1}=\frac{b w_0^2 k^2}{1+k^2}\bigl((1-\eta)^2-1\bigr)+\frac{q^2}{(1-\eta)^2}.
\]
Since $\eta\in(0,\tfrac12]$, we have $\eta^2-1<0$ and $(1-\eta)^2-1=-\eta(2-\eta)<0$. Define the critical values
\begin{equation}\label{ell02}
q_0^2(\eta):=\frac{b w_0^2 k^2}{1+k^2}\,\eta^2(1-\eta^2),\qquad
q_{-1}^2(\eta):=\frac{b w_0^2 k^2}{1+k^2}\,\eta(2-\eta)(1-\eta)^2.
\end{equation}
For $\eta\in(0,\frac{1}{2}]$,  one checks that $q_0^2  \leq q_{-1}^2$,   with equality holding if and only if $\eta=\frac{1}{2}.$
 Then
\(\mu_0>0\) if and only if \(q^2>q_0^2\), and \(\mu_{-1}>0\) if and only if \(q^2>q_{-1}^2\).

From the above analysis and using (\ref{munellxi}) and (\ref{Knxi}), the following lemmas hold true. 

\begin{lemma}
Suppose that $\eta\in(0,\frac{1}{2}]$, $b>0$ and $q>0$.
 \begin{enumerate}[(i)]
\item When $n=1$ or $|n|\ge 2$, we have $(n+\eta)^2-1>0$, thus $K_{n,\eta}=1$.
\item For $q$ large enough, then $\mu_n>0$ for all $n$, and the spectrum is purely imaginary.
\item Only $n=0$ and $n=-1$ may cause sign changes and eigenvalue collisions, which are the only sources of instability.
 \end{enumerate}
 \end{lemma}

\begin{lemma}\label{lemKrein0}
Suppose that $\eta\in(0,\frac{1}{2}]$, $b>0$ and $q>0$.
\begin{enumerate}[(i)]
\item All Krein signatures are positive, spectrum remains purely imaginary for  $q^2 > q_{-1}^2(\eta)$.
\item  When $q_0^2(\eta) < q^2 < q_{-1}^2(\eta)$ and  $\eta\neq\frac{1}{2}$, 
only the Krein signature of  $i\omega_{-1}$ is negative. 
The eigenvalue $i\omega_{-1}$ may collide with eigenvalues carrying positive Krein signature on the imaginary axis.
\item When $0< q^2 < q_0^2(\eta),$ only the Krein signatures of  $i\omega_{-1}$ and $i\omega_{0}$ are negative.  
The eigenvalues $i\omega_{-1}$ and $i\omega_{0}$ may lead to instability under perturbations.
\end{enumerate}
 \end{lemma}

\begin{lemma}[Stability for short-wavelength transverse perturbations]\label{Stabilitysw}
Assume that $\eta\in(0,\frac{1}{2}]$. For any $\epsilon_*>0$, there exists $a_*>0$ such
 that for all $|a|\le a_*$ and $q^2 \ge q_{-1}^2(\eta)+\epsilon_*$, the spectrum of $\mathcal{M}_a(q,\eta)$ is purely imaginary.
\end{lemma}

Based on the analysis of  Lemma \ref{lemKrein0} and Lemma \ref{Stabilitysw},
 the eigenvalues of $\mathcal{M}_0(q,\eta)$ that may give rise to perturbation-induced instability fall into two distinct categories.
   When $q_0^2 \le q^2 \le q_{-1}^2$, only the eigenvalue $i\omega_{-1}(q,\eta)$ is involved.
    When $0<q^2<q_0^2$, both $i\omega_{-1}(q,\eta)$ and $i\omega_0(q,\eta)$ are present.
Our next aim is to isolate these particular eigenvalues from the remaining discrete spectrum, all of whose elements carry a positive
 Krein signature. This separation can be accomplished as long as no collisions occur between the eigenvalues of
  interest and the rest of the point spectrum. 
 In the following, we  analyze the collisions   between $\omega_{-1}(q,\eta)$ and    $\omega_{n}(q,\eta)$ on the interval $(0, q_{-1}^2).$     
\begin{lemma}\label{lemw0wn}
Suppose that $\eta\in(0,\frac{1}{2}]$, $b>0$ and $q>0$. Define 
\begin{equation}\label{ellc2}
q_c^2
:=
\frac{b w_0^2 k^2 \,\eta^2(\eta-1)^2}{1+k^2}
\frac{3 + k^2(\eta^2 - \eta + 1)}
{  (1+k^2) - 3k^2 \eta (1- \eta) }.
\end{equation}
$q_0^2(\eta)$ and $q_{-1}^2(\eta)$ are given by (\ref{ell02}).
Then the following conclusions hold true.
 \begin{enumerate}[(i)]
\item  When $0 \leq q^2 \leq q_{-1}^2,$  collisions cannot occur between
 $\omega_{-1}(q,\eta)$  and  $\omega_n(q,\eta)$ for   $|n| \geq 2.$
\item  Collisions  occur between
 $\omega_{-1}(q,\eta)$  and  $\omega_0(q,\eta)$ if and only if    $ q^2 = q_c^2.$
\item
 When  $ q_0^2 < q^2  \leq q_{-1}^2,$ collisions cannot occur between
 $\omega_{-1}(q,\eta)$  and  $\omega_1(q,\eta)$ 
\item  If $k^2 > \frac{3}{1-\eta^2},$ then collisions may occur between
 $\omega_{-1}(q,\eta)$  and  $\omega_1(q,\eta)$ for    $0< q^2 < q_0^2.$
\item If $k^2 \leq  \frac{3}{1-\eta^2},$ then collisions cannot occur between
 $\omega_{-1}(q,\eta)$  and  $\omega_1(q,\eta)$ for all   $0< q^2 < q_0^2.$
 \end{enumerate}
 \end{lemma}

\begin{proof}

For \( \eta \in (0, \frac{1}{2}] \) and $q>0$, note that the function 
 \( f(x) = \frac{x}{1 + k^2 x^2} (  \frac{b w_0^2 k^2 (x^2 -1)}{1+k^2} + \frac{q^2}{x^2}) \) 
  is odd. Since  $f(x)>0$ for $x \geq 1$,  we have
$ \omega_{n}(q,\eta) < 0 $ for $n \leq -2$, and   $ \omega_n(q,\eta) >0 $  for $n \geq 1$.

Since \[
\omega_{-1}(q,\eta)
= \frac{\eta-1}{1+k^2(\eta-1)^2}
\left(
b w_0^2 \cdot \frac{k^2\big((\eta-1)^2-1\big)}{1+k^2}
+ \frac{q^2}{(\eta-1)^2}
\right),
\]
setting $\omega_{-1}(q,\eta)=0,$ one has $q^2(\eta)=q_{-1}^2(\eta)=\frac{b w_0^2 k^2}{1+k^2}\,\eta(2-\eta)(1-\eta)^2.$
 When $q^2(\eta)<q_{-1}^2(\eta),$  $\omega_{-1}(q,\eta)>0.$ Note that $ \omega_{n}(q,\eta) < 0 $ for $n \leq -2$.
 Thus, when $n\leq -2,$
the collisions cannot occur between
 $\omega_{-1}(q,\eta)$  and  $\omega_n(q,\eta).$  

Consider the difference function
\[
\Phi_n(q^2,\eta):=\omega_n(q,\eta) - \omega_{-1}(q,\eta)= \Psi_n(\eta) + \Omega_n(\eta)\,q^2,
\]
where
\[
\Psi_n(\eta)= \frac{b\,w_0^2\,k^2}{1+k^2}\left[
\frac{(n+\eta)((n+\eta)^2-1)}{1+k^2(n+\eta)^2}
+
\frac{(1-\eta)((1-\eta)^2-1)}{1+k^2(1-\eta)^2}
\right],
\]
\[
\Omega_n(\eta)= \frac{1}{(n+\eta)(1+k^2(n+\eta)^2)}
+
\frac{1}{(1-\eta)(1+k^2(1-\eta)^2)}.
\]
 Since $f(x)=\frac{(x+\eta)((x+\eta)^2-1)}{1+k^2(x+\eta)^2}
+ \frac{(1-\eta)((1-\eta)^2-1)}{1+k^2(1-\eta)^2}$ 
is increasing for $|x+\eta|>\frac{\sqrt{3}}{3},$ which implies  that   $G_n(\eta)$ is increasing   with respect to $n$ for all  $n\geq 2$.
Note that $\Omega_n(\eta)>0$ for $n\geq2.$ Since $f(2)>0$, we have that $\Phi_n(q^2,\eta)>0$.  Thus, collisions do not occur between
 $\omega_{-1}(q,\eta)$  and  $\omega_n(q,\eta)$ for   $n\geq2.$

When $n=0$, then
\[
\Phi_0(q^2,\eta)=\omega_0(q,\eta) - \omega_{-1}(q,\eta)= \Psi_0(\eta) + \Omega_0(\eta)\,q^2.
\]
Obviously, $\Phi_0(q^2,\eta)$ is a linear and monotonic function in $q^2$, $\Phi_0(q_0^2,\eta)<0$ and 
 $\Phi_0(q_{-1}^2,\eta)>0$ for  $0<\eta<\frac{1}{2}$.
Thus, there exists a unique $q_c^2 \in (q_0^2, q_{-1}^2)$ such that $\Phi_0(q_{c}^2,\eta)=0.$
By computation, we have 
\begin{equation*}
q_c^2
=
\frac{b w_0^2 k^2 \,\eta^2(\eta-1)^2}{(1+k^2)}
\frac{3 + k^2(\eta^2 - \eta + 1)}
{  (1+k^2) - 3k^2 \eta (1- \eta) }.
\end{equation*}
Thus, collisions  occur between
 $\omega_{-1}(q,\eta)$  and  $\omega_0(q,\eta)$ if and only if    $ q^2 = q_c^2.$

When $n=1$, 
\[
\Psi_1(\eta)=
\frac{2b\,w_0^2\,k^2}{(1+k^2)}\,
\frac{\eta^2\bigl[3+k^2(\eta^2-1)\bigr]}
{\bigl[1+k^2(1+\eta)^2\bigr]\bigl[1+k^2(\eta-1)^2\bigr]}.
\] 
Obviously, $\Psi_1(\eta)<0$  if and only if $k^2>\frac{3}{1-\eta^2}.$ Note that $\Omega_1(\eta)>0.$ When $k^2>\frac{3}{1-\eta^2},$ 
there exists a unique $q_d^2$ such 
that $\Phi_1(q_d^2,\eta)=0,$ where 
\begin{equation}\label{elld}
  q_d^2=\frac{b\,w_0^2\,k^2}{(1+k^2)}\,
\frac{\eta^2(\eta^2-1)\bigl(3+k^2\eta^2-k^2\bigr)}
{1+k^2+3k^2\eta^2}.
\end{equation}
Direct calculation yields  $ q_d^2<q_0^2.$ Thus, the proof of the lemma is complete. 

\end{proof}

Next, we  analyze the collisions   between $\omega_{0}(q,\eta)$ and    $\omega_{n}(q,\eta).$  
When \(q_0^2(\eta) \le q^2 < \ell_{-1}^2(\eta)\), the Krein signature associated with the eigenvalue \(i\omega_{0}\) is positive. Consequently, \(i\omega_{0}\) cannot lead to instability under small perturbations. We therefore restrict our subsequent analysis to the regime 
 \(0 < q^2 < q_{0}^2(\eta)\).

\begin{lemma}
Assume \(\eta \in(0, \tfrac{1}{2}]\) and \(0<q^{2}<q_{0}^{2}\).
If \(k^{2} \leq 4\), then \(\omega_{0}(q, \eta) \neq \omega_{n}(q, \eta)\) for all  \(n \neq 0\).
Meanwhile, if \(k^{2} \leq3\), then \(\omega_{-1}(q, \eta) \neq\omega_{n}(q, \eta)\) for all  \(n \neq-1\).
\end{lemma}

\begin{proof}
Since
\[
\omega_0(q,\eta)
= \frac{\eta}{1+k^2 \eta^2}
\left(
b w_0^2 \cdot \frac{k^2\big(\eta^2-1\big)}{1+k^2}
+ \frac{q^2}{\eta^2}
\right),
\]
$\omega_0(q,\eta)$  is a linear and monotonic function in $q^2$, 
\(\omega_0(q,\eta)\big|_{q^2=0}<0\), 
 $\omega_0(q,\eta)>0$ for sufficiently large $q^2$.
Thus, there exists a unique $q_0^2$ such that $\omega_0(q_0^2,\eta)=0.$
Solving this zero condition gives $q_{0}^2(\eta)=\frac{b w_0^2 k^2}{1+k^2} \eta^2 (1-\eta^2).$
As a result, \(\omega_0(q,\eta)>0\) for \(q^2>q_0^2(\eta)\) and \(\omega_0(q,\eta)<0\) for \(q^2<q_0^2(\eta)\).

 To detect potential eigenvalue collisions between \(\omega_0\) and  \(\omega_n\), we define the spectral difference function
\begin{align*}
\tilde{\Phi}_{n}\left(q^{2}, \eta\right) =  \omega_{n}(q, \eta)-\omega_{0}(q, \eta)
=   \tilde{\Psi}_{n}(\eta)+q^{2} \tilde{\Omega}_{n}(\eta),
\end{align*}
where
\begin{equation}
 \tilde{\Psi}_{n}(\eta)=  \frac{b w_0^2  k^{2}}{1+k^{2}}\left[\frac{\eta\left(1-\eta^{2}\right)}{1+k^{2} \eta^{2}}+\frac{(n+\eta)\left[(n+\eta)^{2}-1\right]}{1+k^{2}(n+\eta)^{2}}\right], 
\end{equation}
\begin{equation}
\tilde{\Omega}_{n}(\eta) = -\frac{1}{\eta\left(1+k^{2} \eta^{2}\right)}+\frac{1}{(n+\eta)\left[1+k^{2}(n+\eta)^{2}\right]}. 
\end{equation}
Note that  $\tilde{\Phi}_{n}(q^{2}, \eta)$ can be rewritten as 
\begin{equation}
\tilde{\Phi}_{n}(q ^{2},\eta )=\frac {q_{0}^{2}-q ^{2}}{\eta \left( 1+k^{2}\eta ^{2}\right) }+ 
\frac {\alpha (n+\eta )^{2}\left[ (n+\eta )^{2}-1\right] +q ^{2}}{(n+\eta )\left[ 1+k^{2}(n+\eta )^{2}\right] }, \tag{5.8}
\end{equation}
where $\alpha=\frac{b w_0^2 k^{2}}{1+k^{2}}$. Since
$ \frac {q _{0}^{2}-q ^{2}}{\eta \left( 1+k^{2}\eta ^{2}\right) } >0 $ for \(0 \le q^2 < q_{0}^2\), 
$\frac {\alpha (n+\eta )^{2}\left[ (n+\eta )^{2}-1\right] +q ^{2}}{(n+\eta )\left[ 1+k^{2}(n+\eta )^{2}\right] } >0 $ for $n \geq1$,
  we only need to analyze  the case with $n \leqslant-2$. Consider  
\[
g(x)=\frac{q^{2}+\alpha x^{2}\left(x^{2}-1\right)}{x\left(1+k^{2} x^{2}\right)}, \quad  x \leq-\frac{3}{2}.
\]
Then by using
$
q^{2}<q_{0}^{2}=\alpha \eta^{2}\left(1-\eta^{2}\right),~ \eta^{2}\left(1-\eta^{2}\right)<\frac{1}{4},
$
we obtain 
\begin{align*}
g'(x) & =\frac{\alpha\left(k^{2} x^{6}+k^{2} x^{4}+3 x^{4}-x^{2}\right)-q^{2}\left(3 k^{2} x^{2}+1\right)}{x^{2}\left(1+k^{2} x^{2}\right)^{2}} \\
& \geq \frac{\alpha\left[k^{2} x^{6}+k^{2} x^{4}+3 x^{4}-x^{2}-\eta^{2}\left(1-\eta^{2}\right)\left(3 k^{2} x^{2}+1\right)\right]}{x^{2}\left(1+k^{2} x^{2}\right)^{2}}>0 
\end{align*}
 for all $ x \leq-\frac{3}{2}.$
Thus,  $\frac {\alpha (n+\eta )^{2}\left[ (n+\eta )^{2}-1\right] +q ^{2}}{(n+\eta )\left[ 1+k^{2}(n+\eta )^{2}\right] }$ is
 increasing with $n$ for $n \leqslant-2$, which yields $\tilde{\Phi}_{n}(q^{2}, \eta) \leq\tilde{\Phi}_{-2}(q^{2}, \eta)$ for $n \leqslant-2$.
Note that $\tilde{\Omega}_{n}(\eta)<0$ for $n \leqslant-2$. Direct computation yields  
\[
\tilde{\Psi}_{-2}(\eta)=\frac{2 \alpha(1-\eta)^{2}\left[k^{2} \eta(2-\eta)-3\right]}{\left(1+k^{2} \eta^{2}\right)\left[1+k^{2}(\eta-2)^{2}\right]}
\]
and
\[
\frac {1}{2}\tilde {\Phi}_{-2}(q ^{2},\eta )=\frac {\left( 1+k^{2}\left[ 1+3(1-\eta )^{2}\right] \right) q ^{2}-\left[ k^{2}\eta (2-\eta )-3\right]
 q_{-1}^{2}}{\eta (\eta -2)(1+k^{2}\eta ^{2})\left[ 1+k^{2}(\eta -2)^{2}\right] }.
\]
If \(k^{2} \leq \dfrac{3}{\eta(2-\eta)}\), then \(\tilde{\Phi}_{-2}(q^{2}, \eta)<0\),
 it follows that \(\tilde{\Phi}_{n}(q^{2}, \eta)<0\) holds for all  \(n \leq -2\).
 This excludes collisions between \(\omega_0\) and all modes with \(n\neq 0,-1\).

For \(\eta\in(0,\frac12]\), one verifies \(\frac{3}{\eta(2-\eta)}\ge 4\). Thus the condition \(k^2\le 4\) guarantees \(k^{2} \leq \frac{3}{\eta(2-\eta)}\) uniformly in \(\eta\), so that \(\omega_0\) never collides with any other eigenvalue. In addition, \(\frac{3}{1-\eta^2}>3\) over the considered \(\eta\)-interval. From Lemma \ref{lemw0wn}, collisions involving \(\omega_{-1}\) can only take place for \(k^2>\frac{3}{1-\eta^2}\). 
Consequently, \(k^2\le 3\) ensures that \(\omega_{-1}\) has no collision with any other mode for \(0<q^2<q_0^2\). 
Finally, the unique crossing between \(\omega_0\) and \(\omega_{-1}\) occurs at \(q_c^2\ge q_0^2\), 
which lies outside the present parameter regime. The proof is  complete.

\end{proof}

\subsection{ Spectrum of the perturbed operator $\mathcal{M}_a(q,\eta)$ for  $q_0^2(\eta) < q^2 < q_{-1}^2(\eta) + \epsilon_*$}

Lemma \ref{lemKrein0} indicates that multiple eigenvalue collisions may emerge for \(0<q^2<q_0^2(\eta)\), 
which substantially complicates the analysis of the perturbed system. Therefore
 we  restrict our discussion to the parameter range \(q_0^2(\eta) < q^2 < q_{-1}^2(\eta) + \epsilon_*\).

\begin{theorem}[Transverse instability under finite-wavelength transverse perturbation]\label{thm:NIbwd}
  Assume that $\eta \in \big(0,\frac12\big]$, $q_0^2(\eta)$  and  $q_c^2(\eta)$ are represented by (\ref{ell02}) and  (\ref{ellc2}), 
  respectively. Define
   \begin{equation}\label{Arep}
A := (1+b)\eta + k^2 + \frac{bk^2}{2}\eta(\eta-1) + k^2(\eta-1)^3,
\end{equation}
   \begin{equation}\label{Brep}
B :=\frac{ (1-\eta)\Big( 2(1+b) + 2k^2(1+\eta+\eta^2) - bk^2\eta\Big )}{2},
\end{equation}
  When $AB > 0$, for any $|a| \ll 1$,  there exist $\epsilon_{\text{crit}}(\eta)>0$ with 
\begin{equation}
\epsilon_{\text{crit}}(\eta) :=
\frac{2|a|\,|w_0|\, \eta(1-\eta)\,
\sqrt{AB\,(1+k^2\eta^2)\bigl(1+k^2(\eta-1)^2\bigr)}}
{1 + k^2(3\eta^2-3\eta+1)}
\end{equation} 
such that
  \begin{enumerate}[(i)]
            
\item if $|q^2 - q_c^2| > \epsilon_{\text{crit}}(\eta) $, then the entire spectrum of $\mathcal{M}_a(q,\eta)$ lies purely on the imaginary axis.

\item if $|q^2 - q_c^2| < \epsilon_{\text{crit}}(\eta) $, then $\mathcal{M}_a(q,\eta)$ admits a pair of complex eigenvalues
 with nonzero real parts of opposite signs, while all remaining spectral points are purely imaginary. 
\end{enumerate}
In the alternative case where $AB < 0$, the spectrum of $\mathcal{M}_a(q,\eta)$ is entirely purely imaginary for
 all $q^2$ satisfying $q_0^2 \le q^2 \le q_c^2 + \epsilon_*$.
\end{theorem}

\begin{proof}

For sufficiently small $|a|,\epsilon_*>0$, the spectrum of $\mathcal{M}_a(q,\eta)$ decomposes as
\[
\mathrm{spec}\big(\mathcal{M}_a(q,\eta)\big)
= \mathrm{spec}_0\big(\mathcal{M}_a(q,\eta)\big) \cup \mathrm{spec}_1\big(\mathcal{M}_a(q,\eta)\big),
\]
where $\mathrm{spec}_0$ is a two-dimensional spectral subspace obtained by analytically continuing the unperturbed modes
  $i\omega_0(q,\eta)$ and  $i\omega_{-1}(q,\eta)$ under small perturbations in the parameter $a$,
and $\mathrm{spec}_1$ consists of purely imaginary eigenvalues uniformly separated from $\mathrm{spec}_0$.
By selecting sufficiently small positive constants \(\epsilon_*\) and \(a_*\), 
the above spectral decomposition remains valid for all parameters satisfying \(q_0^2 \le q^2 \le q_c^2+\epsilon_*\) and \(|a|\le a_*\).
 This reduces our remaining task to examining the spatial distribution of the two eigenvalues belonging to \(\mathrm{spec}_0(\mathcal{M}_a(q,\eta))\).

Lemma \ref{lemw0wn} establishes that for \(q_0^2 \le q^2 \le q_c^2+\epsilon_*\)
 with \(q\) outside a small neighbourhood of \(q_c\), the two modes \(i\omega_{-1}(q,\eta)\) 
 and \(i\omega_0(q,\eta)\) are simple. Moreover, there exists a constant \(\hat{c}>0\) such that
\(\big|i\omega_{-1}(q,\eta)-i\omega_0(q,\eta)\big|\ge \hat{c}.\)
When a is taken sufficiently close to zero, the simplicity of this eigenvalue pair is preserved within the spectrum of \(\mathcal{M}_a(q,\eta)\).
 Given the imaginary-axis symmetry of the spectrum of \(\mathcal{M}_a(q,\eta)\), both eigenvalues stay purely imaginary whenever \(q\) 
 lies away from \(q_c\).

For $q=q_c(\eta)$, the unperturbed operator $\mathcal{M}_0(q_c,\eta)$ has two purely imaginary eigenvalues that coalesce
\[
\omega_{-1}(q_c,\eta)=\omega_0(q_c,\eta) := \omega_*(\eta),
\]
with corresponding eigenvectors
\[
\phi_0^0=1, \qquad   \phi_0^{-1}=e^{-iz},
\]
which span the collision subspace $\mathcal{X}_0=\operatorname{span}\{1,e^{-iz}\}$. 
For sufficiently small $a$, this subspace persists as $\mathcal{X}_a(q_c,\eta)$ 
with an analytic basis satisfying $\phi_a^0=1+O(a)$, $\phi_a^{-1}=e^{-iz}+O(a)$.

In the following, we compute the matrix representation of $\mathcal{M}_a(q_c,\eta)$ in the basis $\{\phi_a^0,\phi_a^{-1}\}$.
Direct computation yields that
\[
\begin{aligned}
\mathcal{M}_a(q,\eta) - \mathcal{M}_0(q,\eta)
&= \bigl(1-k^2 (\partial_z + i\eta)^2\bigr)^{-1}
\Bigl[
(c - c_0)((\partial_z + i\eta) - k^2 (\partial_z + i\eta)^3)
- 2(1+b)w w_z \\
&\quad - (1+b)(w^2 - w_0^2)(\partial_z + i\eta)
+ bk^2\bigl(w_z w_{zz} + w w_{zz}(\partial_z + i\eta) + w w_z (\partial_z + i\eta)^2\bigr) \\
&\quad + k^2(w^2 - w_0^2)(\partial_z + i\eta)^3 + 2k^2w w_{zzz}
\Bigr].
\end{aligned}
\]
From Lemma \ref{lem2.1wc}, with  $w=w_0 + a\cos z+O(a^2)$ and $c-c_0=O(a^2)$, we get
\[
\mathcal{M}_a(q,\eta)=\mathcal{M}_0(q,\eta)+a\mathcal{B}+O(a^2),
\]
where
\[
\begin{aligned}
\mathcal{B}
&= w_0 \bigl(1-k^2 (\partial_z + i\eta)^2\bigr)^{-1}
\Bigl[
2(1+b+k^2)\sin z
- 2(1+b)\cos z\,(\partial_z + i\eta)
- b k^2 \cos z\,(\partial_z + i\eta)\\
&~~~ - b k^2 \sin z\,(\partial_z + i\eta)^2
+ 2k^2 \cos z\,(\partial_z + i\eta)^3
\Bigr].
\end{aligned}
\]
Since
\[
(1-k^2(\partial_z + i\eta)^2)^{-1} [2(1+b+k^2)\sin z e^{-iz}]
= -\frac{i(1+b+k^2)}{1+k^2\eta^2} + \frac{i(1+b+k^2)}{1+k^2(\eta-2)^2}e^{-2iz},
\]

\[
(1-k^2(\partial_z + i\eta)^2)^{-1}  [-2(1+b)\cos z\,(\partial_z + i\eta)\,e^{-iz}]
= -\frac{i(1+b)(\eta-1)}{1+k^2\eta^2}  -\frac{i(1+b)(\eta-1)}{1+k^2(\eta-2)^2}e^{-2iz},
\]

\[
(1-k^2(\partial_z + i\eta)^2)^{-1} [-bk^2\cos z\,(\partial_z + i\eta)\,e^{-iz}]
= -\frac{i bk^2(\eta-1)}{2(1+k^2\eta^2)}  -\frac{i bk^2(\eta-1)}{2(1+k^2(\eta-2)^2)}e^{-2iz},
\]

\[
(1-k^2(\partial_z + i\eta)^2)^{-1} [-bk^2\sin z\,(\partial_z + i\eta)^2\,e^{-iz}]
= -\frac{i bk^2(\eta-1)^2}{2(1+k^2\eta^2)} +\frac{i bk^2(\eta-1)^2}{2(1+k^2(\eta-2)^2)}e^{-2iz},
\]

\[
(1-k^2(\partial_z + i\eta)^2)^{-1} [2k^2\cos z\,(\partial_z + i\eta)^3\,e^{-iz}]
 = -\frac{i k^2(\eta-1)^3}{1+k^2\eta^2}  -\frac{i k^2(\eta-1)^3}{1+k^2(\eta-2)^2}e^{-2iz},
\]
 we obtain
\[
\mathcal{B} e^{-iz}
=
w_0\left[
-\frac{i A}{1+k^2\eta^2}
+
\frac{ i \tilde{A}}{1+k^2(\eta-2)^2} e^{-2iz}
\right],
\]
where
\begin{equation*}
A = (1+b)\eta + k^2 + \frac{bk^2}{2}\eta(\eta-1) + k^2(\eta-1)^3,
\end{equation*}
\[
\tilde{A} = 
(1+b+k^2) - (1+b)(\eta-1) - \frac{bk^2}{2}(\eta-1) + \frac{bk^2}{2}(\eta-1)^2 - k^2(\eta-1)^3.
\]
Since the \(e^{-2iz}\) term is orthogonal to the constant function, we have
\[
\langle 1,\mathcal{B}e^{-iz}\rangle = -\frac{i w_0 A}{1+k^2\eta^2}.
\]
Since
\[
\begin{aligned}
& \Bigl[
2(1+b+k^2)\sin z
- 2(1+b)\cos z\,(\partial_z + i\eta)
- b k^2 \cos z\,(\partial_z + i\eta)
- b k^2 \sin z\,(\partial_z + i\eta)^2
+ 2k^2 \cos z\,(\partial_z + i\eta)^3
\Bigr] \cdot 1\\
&= \bigl[2(1+b+k^2) + bk^2\eta^2\bigr]\sin z
- i\eta\bigl[2(1+b)+bk^2+2k^2\eta^2\bigr]\cos z,
\end{aligned}
\]
by using
$\sin z = \frac{e^{iz}-e^{-iz}}{2i}$ and  $\cos z = \frac{e^{iz}+e^{-iz}}{2},$
we have
\[
\begin{aligned}
\mathcal{B}\cdot 1
&=
w_0\Bigg[
-\frac{i}{2} \frac{(1+\eta)\Big( 2(1+b)+2k^2 + k^2\eta(b-2)+2k^2\eta^2\Big) }{1+k^2(1+\eta)^2} e^{iz}\\
&+
\frac{i}{2} \frac{(1-\eta)\Big( 2(1+b)+2k^2 ( 1 + \eta +\eta^2 ) - b k^2\eta\Big) }{1+k^2(1-\eta)^2} e^{-iz}
\Bigg].
\end{aligned}
\]
Thus
\[
\langle e^{-iz}, \mathcal{B}\cdot 1\rangle
=
i \frac{w_0}{1+k^2(1-\eta)^2} B,
\]
where
\begin{equation*}
B =\frac{ (1-\eta)\Big( 2(1+b) + 2k^2(1+\eta+\eta^2) - bk^2\eta\Big )}{2}.
\end{equation*}
Note that if $q^2=q_c^2,$ then
\[
\omega_0(q,\eta)=\omega_{-1}(q,\eta)=\omega_*(\eta),
\]
where
\[
q_c^2
=
\frac{b w_0^2 k^2 \,\eta^2(\eta-1)^2}{(1+k^2)}
\frac{3 + k^2(\eta^2 - \eta + 1)}
{  (1+k^2) - 3k^2 \eta (1- \eta) }.
\]
when we set $\epsilon = q^2 - q_c^2$, then we have
\[
\omega_0(q,\eta) = \omega_*(\eta) + \frac{\epsilon}{\eta(1+k^2\eta^2)}
\]
and
\[
\omega_{-1}(q,\eta) = \omega_*(\eta) + \frac{\epsilon}{(\eta-1)\bigl(1+k^2(\eta-1)^2\bigr)}.
\]
Thus the full matrix $M_a(q,\eta)$ in the chosen basis becomes
\[
\begin{aligned}
M_a(q,\eta)&=
\begin{pmatrix}
i\omega_*+i\dfrac{\epsilon}{\eta(1+k^2\eta^2)}
&
-\frac{i w_0 }{1+k^2\eta^2} A \,a\\[1.2ex]
i \frac{w_0}{1+k^2(1-\eta)^2} B \,a
&
i\omega_*+i\dfrac{\epsilon}{(\eta-1)[1+k^2(\eta-1)^2]}
\end{pmatrix}\\
&+O(a^2+|a\epsilon|).
\end{aligned}
\]
We seek eigenvalues of $M_a(q,\eta)$ in the form
\[
\lambda = i\omega_* + iX,
\]
by substituting \(\lambda\) into \(\det(\lambda I - M_a)=0\) and keeping terms up to \(O(a^2 |\epsilon|+|a|\epsilon^2+|a|^3)\), we obtain
\[
\begin{aligned}
&X^2 - \epsilon\left[
\frac{1}{\eta(1+k^2\eta^2)}
+
\frac{1}{(\eta-1)(1+k^2(\eta-1)^2)}
\right] X\\
&+
\frac{\epsilon^2}{\eta(\eta-1)(1+k^2\eta^2)(1+k^2(\eta-1)^2)}
+
\frac{w_0^2 A B\,a^2}{(1+k^2\eta^2)(1+k^2(\eta-1)^2)}
= 0.
\end{aligned}
\]
 We compute the discriminant  as
\[
\begin{aligned}
\Delta_a(\epsilon,\eta) = \epsilon^2 \left[
\frac{1}{\eta(1+k^2\eta^2)}
-
\frac{1}{(\eta-1)(1+k^2(\eta-1)^2)}
\right]^2  - 
\frac{4 w_0^2 A B\,a^2}{(1+k^2\eta^2)(1+k^2(\eta-1)^2)}.
\end{aligned}
\]

When $AB>0,$ the boundary between stability and instability is given by $\Delta_a(\epsilon,\eta)  = 0$. Solving for $\epsilon$ 
gives the critical detuning
\[
\epsilon_{\text{crit}}(\eta) =
\frac{2|a|\,|w_0|\, \eta(1-\eta)\,
\sqrt{AB\,(1+k^2\eta^2)\bigl(1+k^2(\eta-1)^2\bigr)}}
{1 + k^2(3\eta^2-3\eta+1)}.
\]
For any $a$ sufficiently small,  when \(|\epsilon| > \epsilon_{\text{crit}}(\eta)\), we have \(\Delta_a(\epsilon,\eta) \ge 0\).
   When \(|\epsilon| < \epsilon_{\text{crit}}(\eta) \), we have \(\Delta_a(\epsilon,\eta) < 0\).
This implies that, for \(q_0^2 \le q^2 \le q_c^2 + \epsilon_*\), 
when \(|q^2 - q_c^2| >  \epsilon_{\text{crit}}(\eta) \), the two eigenvalues of \(\mathcal{M}_a(q, \eta)\) are purely imaginary.
When \(|q^2 - q_c^2| <  \epsilon_{\text{crit}}(\eta) \), the two eigenvalues of \(\mathcal{M}_a(q, \eta)\) are
 complex with opposite nonzero real parts.

 When \(AB < 0\), we have \(\Delta_a(\epsilon,\eta) > 0\), and this implies that the two eigenvalues of \(\mathcal{M}_a(q, \eta)\)
  are purely imaginary for \(q_0^2 \le q^2 \le q_c^2 + \epsilon_*\).
\end{proof}

Next, we consider the sign analysis of $AB$.
\begin{lemma}[Unified criterion for the sign of $AB$ and global conditions]
    \label{lem:AB-full-criterion}
    Let $b>0$, $0<\eta\le \tfrac12$, $k>0$. Define
    \[
    f_A(\eta):=\eta^2+\frac{b-6}{2}\eta+\frac{6-b}{2},\qquad 
    g_B(\eta):=b\eta-2(1+\eta+\eta^2),
    \]
    where $A$ and $B$ are given by \eqref{Arep} and \eqref{Brep}, respectively.
    For $f_A(\eta)<0$ and $g_B(\eta)>0$, set
    \[
    \alpha_A(\eta):=-\frac{1+b}{f_A(\eta)},\qquad 
    \alpha_B(\eta):=\frac{2(1+b)}{g_B(\eta)}.
    \]
    We adopt the convention that the corresponding threshold equals $+\infty$ if its definition condition fails.

    \begin{enumerate}
        \item[(1)]     $AB>0$  for fixed $\eta$ if and only if one of the following mutually exclusive cases holds:
        \begin{enumerate}[(i)]
            \item $f_A(\eta)\ge 0$ and $g_B(\eta)\le 0$;
            \item $f_A(\eta)\ge 0$, $g_B(\eta)>0$ and $k^2 < \alpha_B(\eta)$;
            \item $f_A(\eta)<0$, $g_B(\eta)\le 0$ and $k^2 < \alpha_A(\eta)$;
            \item $f_A(\eta)<0$, $g_B(\eta)>0$ and either $k^2 < \min\{\alpha_A,\alpha_B\}$ or $k^2 > \max\{\alpha_A,\alpha_B\}$.
        \end{enumerate}
        
        Meanwhile, $AB<0$ for fixed $\eta$ if and only if one of the following mutually exclusive cases holds:
        \begin{enumerate}[(i)]
            \item $g_B(\eta)>0$ and $\alpha_B(\eta) < k^2 < \alpha_A(\eta)$; 
            \item $f_A(\eta)<0$ and $\alpha_A(\eta) < k^2 < \alpha_B(\eta)$.
        \end{enumerate}

        \item[(2)] \textit{Global uniform positive condition over $\eta\in(0,\tfrac12]: $}
        \[
        AB>0,\;\forall \eta\in(0,\tfrac12]
        \iff
        \begin{cases}
        k>0, & 0<b\le 6,\\[4pt]
        0<k^2 < \dfrac{2(1+b)}{b-6}, & b>6.
        \end{cases}
        \]

        \item[(3)] 
        
        There exist no parameters $b>0,k>0$ such that $AB<0$ holds for all $\eta\in(0,\tfrac12]$.

        \item[(4)] 
        
        There exists at least one $\eta\in(0,\tfrac12]$ satisfying $AB<0$ if and only if
        \[
        b>6 \quad\text{and}\quad k^2 > \frac{2(1+b)}{b-6}.
        \]
    \end{enumerate}

    \end{lemma}

    \begin{proof}
Note that
    \[
    A = \eta \big[(1+b) + k^2 f_A(\eta)\big], \quad
    B = \frac{1-\eta}{2}\big[2(1+b) - k^2 g_B(\eta)\big].
    \]
    Since $\eta>0$ and $1-\eta>0$ on $(0,\tfrac12]$, the sign of $A$ coincides with $(1+b)+k^2f_A(\eta)$, and the sign of $B$ coincides with $2(1+b)-k^2g_B(\eta)$. Combining sign variations of $A$ and $B$ directly yields all pointwise criteria in item (1).

    Next we analyze the two auxiliary quadratics $f_A(\eta)$ and $g_B(\eta)$.
    The quadratic $f_A(\eta)$ is upward-opening with vertex $\eta_0=\frac{6-b}{4}$.  Moreover,
 $f_A(\eta_0) = \frac{(6-b)(b+2)}{16}$, $f_A(0) = \frac{6-b}{2}$ and $f_A(\frac{1}{2}) = \frac{7-b}{4}$.
        When $0<b\le 6$, $f_A(\eta)\ge 0$ for all $\eta\in(0,\tfrac12]$, hence $A>0$ uniformly.    
     When  $6<b<7$, $f_A$ has a single zero $\eta_A\in(0,\tfrac12)$, negative on $(0,\eta_A)$ and positive on $(\eta_A,\tfrac12]$.  
        Besides, $\alpha_A(\eta)$ is strictly increasing on $(0,\eta_A)$ with $\lim_{\eta\to0^+}\alpha_A(\eta)=\frac{2(1+b)}{b-6}$.
     When   \(b\ge 7\),  \(f_A(\eta)<0\) for all \(\eta\in(0,\tfrac12]\), with equality only at \(\eta=\tfrac12\) when \(b=7\). 
    The concave quadratic $g_B(\eta)$ has vertex $\eta_1=\frac{b-2}{4}$. Note that 
     $g_B(\eta_1) = \frac{(b-6)(b+2)}{8}$, \(g_B(0)=-2\) and \(g_B(\frac{1}{2})=\frac{b-7}{2}\).
    For $b\le7$,  $g_B(\eta) \leq 0$ for all  $\eta\in(0,\tfrac12]$, with equality only at \(\eta=\tfrac12\) when \(b=7\),
     it yields $B>0$ for all  $\eta\in(0,\tfrac12]$.
     For $b>7$, $g_B$ admits a unique zero $\eta_B\in(0,\tfrac12)$, negative on $(0,\eta_B)$ and positive on $(\eta_B,\tfrac12]$. 
     Meanwhile, $\alpha_B(\eta)$ is strictly decreasing on $(\eta_B,\tfrac12]$, with $\alpha_B(\tfrac12)=\frac{4(1+b)}{b-7}$.

    We now prove the global criteria.
  
      For $0<b\le6$, $A$ and $B$ are always positive. 
For \(6<b<7\), \(B>0\) holds for all \(\eta\in(0,\tfrac12]\), while \(A>0\) is valid on \(\eta\in(\eta_A,\tfrac12]\).
 If \(k^2< \frac{2(1+b)}{b-6}\), then \(A>0\) holds uniformly for \(\eta\in(0,\tfrac12]\), which further implies \(AB>0\). 
 For \(b\ge 7\), \(A>0\) under the condition \(k^2<\frac{2(1+b)}{b-6}\) and \(B>0\) whenever \(k^2<\frac{4(1+b)}{b-7}\). 
 Since $\frac{2(1+b)}{b-6}<\frac{4(1+b)}{b-7}$ for all \(b\ge 7\), the single constraint \(k^2<\frac{2(1+b)}{b-6}\) 
 simultaneously guarantees both \(A>0\) and \(B>0\).  This establishes item (2).

       Uniform opposite signs require $A<0$ and $B>0$ over the whole interval, which needs contradictory bounds $k^2>\frac{4(1+b)}{b-7}$ and $k^2<\frac{4(1+b)}{b-7}$. Hence such parameters do not exist. This establishes item (3).

     If $b>6$ and $k^2>\frac{2(1+b)}{b-6}$, choose sufficiently  small  $\eta>0$, then $A<0$ and $B>0$, which yields $AB<0$. 
     If $b\le6$, $A>0$ and $B>0$ always hold, $AB<0$ never occurs. This proves item (4).
 
    The proof is complete.
    \end{proof}

\noindent{\bf Acknowledgements}

This work is supported by 
the National Natural Science Foundation of China (No.12571172, No.12471158, No.12475049)
 and by the Natural Science Foundation of Hunan Province (No.2026JJ81121).\\

\end{document}